\documentclass{article}
\usepackage[a4paper, total={6in, 8in}]{geometry}
\usepackage{amsmath, amsthm, amssymb, amsfonts, mathtools, bbold, stmaryrd}
\usepackage{tikz-cd}
\usepackage[utf8]{inputenc}
\usepackage[T1]{fontenc}
\usepackage{mdframed}
\usepackage[english]{babel}
\usepackage{newlfont}
\usepackage{fancyhdr}
\usepackage{graphicx}
\usepackage{latexsym}
\usepackage[utf8]{inputenc}
\usepackage{enumerate}
\usepackage{enumitem}
\usepackage{nicefrac}
\usepackage{pgf,tikz}
\usepackage{mathrsfs}
\usetikzlibrary{arrows}
\usepackage[all,cmtip]{xy}
\usepackage{caption}
\usepackage{xcolor}
\usepackage{hyperref}
\usepackage{authblk}

\newenvironment{AB}{\color{red}}{\color{black}}
\newenvironment{pp}{\color{blue}}{\color{black}}

\theoremstyle{plain}
\newtheorem{theorem}{Theorem}[section]
\newtheorem*{theorem*}{Theorem}
\newtheorem{proposition}[theorem]{Proposition}
\newtheorem{corollary}[theorem]{Corollary}
\newtheorem{lemma}[theorem]{Lemma}

\theoremstyle{definition}
\newtheorem{definition}[theorem]{Definition}

\newtheorem{notation}[theorem]{Notation}

\newtheorem{remark}[theorem]{Remark}
\newtheorem{example}[theorem]{Example}

\newtheorem{claim}{Claim}[theorem]

\title{Identifiability of rank-3 tensors}
\date{}
\author{Edoardo Ballico \thanks{edoardo.ballico@unitn.it}}
\author{Alessandra Bernardi \thanks{alessandra.bernardi@unitn.it}}
\author{Pierpaola Santarsiero \thanks{p.santarsiero-1@unitn.it}}
\affil{Dipartimento di Matematica, Univ. Trento, Italy}

\begin{document}
\maketitle

\begin{abstract}
Rank-2 and rank-3 tensors are almost all identifiable with only few exceptions. We classify them all together with the dimensions and the structures of all the sets evincing the rank. 
\end{abstract}
\section*{Introduction}

Identifiability of tensors is one of the most active research areas both in pure mathematics and in applications. The core of the problem is being able to understand if a given tensor $T\in \mathbb{C}^{n_1+1} \otimes \cdots \otimes \mathbb{C}^{n_k+1}$ can be decomposed in a unique way as a sum of pure tensors:
$$T=\sum_{i=1}^r v_{1,i}\otimes \cdots \otimes v_{k,i},$$
with $v_{j,i}\in \mathbb{C}^{n_j+1}$, for $j=1, \ldots , k$. Of course the minimum $r$ realizing the above expression is a crucial value and it is called the \emph{rank} of $T$.

From the applied point of view, identifiability in tensor decomposition arises naturally in numerous areas, we quote as examples Phylogenetics, Quantum Physics, Complexity Theory and Signal Processing  (cf. eg. \cite{ar, orus, quantumgeo, hjn, bv, bercar, bcs, Land2, rlz, DDL1, DDL2, DDL3, DDL4}).

From the pure mathematical point of view, being able to understand if a tensor is identifiable is a very elegant problem that goes back to Kruskal \cite{Kruskal} and finds more modern contributions with the language of Algebraic Geometry and Multilinear Algebra in eg. \cite{cc1, cc2, C.O., cov, cov2, id1, id2, gm, mms, cm, acm, acv, bbc, bbcc, bccgo}. Except for very few contributions \cite{Kruskal, DDL1, js} which work  for certain specific classes of given tensors, all the others regard the identifiability of generic tensors of certain rank. From the computational point of view, as far as we know, the unique algorithm dealing with the identifiability of any given tensor is a numerical one developed in Bertini \cite{bertini} in \cite{bdhm}.

Dealing with tensors of given rank $r$ brings the problem into the setting of 
$r$-th secant varieties of Segre varieties (cf. Definition \ref{secant}) namely the closure (either Zariski or Euclidean closures can be used for this definition if working over $\mathbb{C}$) of the set of tensors of rank smaller or equal than $r$. 
Knowing if a generic tensor of certain rank is identifiable gives an indication regarding the behaviour of specific tensors of the same rank. Namely, the dimension of the set $\mathcal{S}(Y,T)$ of rank-1 tensors computing the rank of a specific tensor $T$ (cf. Definition \ref{SYq}) can only be bigger or equal than the dimension of $\mathcal{S}(Y,q)$ where $q$ is a generic tensor of rank equal to the rank of $T$ (this will be explained in Remark \ref{allP14} for the specific case of rank-3 tensors $T\in (\mathbb{C}^2)^{\otimes 4}$, but it is a well known general fact for which we refer \cite[Cap II, Ex 3.22, part (b)]{Hart}). Since the cases in which generic tensors of fixed rank are not-identifiable are rare (cf. eg. \cite{bccgo, cm, gm, lf, id1, CGGrank,  CGGtecnica, CGG, C.O.}), the knowledge of generic tensors' behaviour doesn't help in all the applied problems where the ken  of a specific tensor modeling certain precise samples is required. 

In the present manuscript we present a systematic study of the identifiability of a given tensor starting with those of ranks 2 and 3. We give a complete classification of these first cases: we describe the structures  and the dimensions of all the sets evincing the rank. In terms of generic tensors of rank either 2 or 3, everything was already well known form \cite{AOP, BalBer3, CGGrank, CGG, C.O., cov, Kruskal, CMR}. What it was missing was the complete classification for all the tensors of those ranks.

In Proposition \ref{sigma20} we show that rank-2 tensors $T$ are always identifiable except if $T$ is a $2\times 2$ matrix.
Our \textbf{main Theorem \ref{main_theorem}} states that a rank-3 tensor $T$ is identifiable except if
\begin{enumerate}
    \item $T$ is a $3\times 3$ matrix and $\dim(\mathcal{S}(Y,T))=6$;
    
    \item there exist $v_1,v_2,v_3\in \mathbb{C}^2$ s.t. $T\in \mathbb{C}^2\otimes v_2 \otimes v_3+ v_1 \otimes \mathbb{C}^2 \otimes v_3+ v_1 \otimes v_2\otimes \mathbb{C}^2$ and $\dim(\mathcal{S}(Y,T))\geq 2$;  
    
    \item  $T\in (\mathbb{C}^2)^{\otimes 4}$ and $\dim(\mathcal{S}(Y,T))\geq 1$;
    
    \item  $T\in\mathbb{C}^3\otimes \mathbb{C}^2\otimes \mathbb{C}^2$ and it is constructed as in Example \ref{caso3}. 
    In this case $\dim(\mathcal{S}(Y,T))=3$;
    
    \item $T\in\mathbb{C}^3\otimes \mathbb{C}^2\otimes \mathbb{C}^2$ and it is constructed as in Example \ref{caso4}. 
    In this case $\mathcal{S}(Y,T)$ contains two different 4-dimensional families;
    
    \item $T \in \mathbb{C}^{m_1}\otimes \mathbb{C}^{m_2}\otimes (\mathbb{C}^2)^{k-2} $, where $k\geq 3$ and $m_1,m_2\in \{2,3 \} $. In this case $\dim(\mathcal{S}(Y,T))\geq2$ and $ T$ is constructed as in Proposition \ref{x1.1}. If  $m_1+m_2+k \geq 6$ then $\dim(\mathcal{S}(Y,T))=2$.
 
\end{enumerate}

The paper is organized as follows. After the preliminary Section \ref{Prelim} where we introduce the notation and the main ingredients needed for the set up, we can immediately show the identifiability of rank-2 tensors in Section \ref{Sec:id:2}. In Section \ref{sec.ex} we explain in details the examples where the non-identifiability of a rank-3 tensor arises. In Sections \ref{section:5points} and \ref{section:6points} we show that the examples of the previous section are the only possible exceptions to non-identifiability of a rank-3 tensor. Section \ref{sec.last} is actually devoted to collect all the information needed (but actually already proved at that stage) to conclude the proof of our main Theorem \ref{main_theorem}.

\section{Preliminaries and Notation}\label{Prelim}

We will always work over an algebraically closed field $\mathbb{K}$ of characteristic $0$.

\begin{definition}
Let $X \subset \mathbb{P}^N $ be a non-degenerate projective variety, the $X $-\emph{rank} $r_{X}(q)$  of a point $ q \in \langle X \rangle$ is the minimal cardinality of a finite set $S\subset X $ such that $q \in \langle S \rangle $. 
\end{definition}

\begin{notation}
Let $A\subset \mathbb{P}^N$ be any subset. With an abuse of notation we denote by $\langle A \rangle$ the projective space spanned by $A$.

\end{notation}

Let $V_1, \ldots , V_k$ be vectors spaces of 
dimension $n_1+1, \ldots, n_k+1$ respectively, the \emph{Segre variety} is the image of the following embedding: $$\nu: \mathbb{P}(V_1)\times \cdots \times \mathbb{P}(V_k) \to  \mathbb{P}(V_1 \otimes \cdots \otimes V_k)$$
$$([v_1], \ldots  , [v_k])\mapsto [v_1 \otimes \cdots \otimes v_k]$$

\begin{notation}\label{pi}We denote by $ Y$ the multiprojective space $$Y:=\mathbb{P}^{n_1}\times \cdots \times \mathbb{P}^{n_k}$$ and by $X$ the image of $ Y$ via Segre embedding, i.e. $X=\nu(Y) $.
\\
We denote the projection on the $i $-th factor as $$\pi_i: Y \longrightarrow \mathbb{P}^{n_i} .$$
The space  corresponding to forget the $i $-th factor in the multiprojective space $Y $ is denoted by $Y_i$: $$Y_i:= \mathbb{P}^{n_1}\times \cdots \times \widehat{\mathbb{P}^{n_i}}\times \cdots \times \mathbb{P}^{n_k}.$$ 
With $\nu_i : Y_i \longrightarrow \mathbb{P}^{N'} $ we denote the corresponding Segre embedding, in particular $X_i:=\nu(Y_i)$.
\\
The projection on all the factors of $Y$ but the $i $-th one is denoted with $\eta_i$:
$$\eta_i : Y \longrightarrow Y_i .$$ Obviously all fibers of $\eta _i$ are isomorphic to $\mathbb{P}^{n_i}$.
\end{notation}

\begin{definition}\label{SYq}
For any $q \in \mathbb{P}^N $, $\mathcal{S}(Y,q) $  denotes   the set of all subsets $ A \subset Y$ such that $ \sharp(A) =r_X(q) $ and $q \in \langle \nu(A) \rangle $ and we will say that if $A \in \mathcal{S}(Y,q)$, then $A$ \emph{evinces} the rank of $q$. Moreover we say that $q\in \langle X \rangle$ is \emph{identifiable} if $\sharp\mathcal{S}(Y,q) =1$. 
\end{definition}

\begin{notation}\label{epsiloneco}
Sometimes we will also use the following multi-index notations: for $1 \leq i \leq k $, $\varepsilon_i=(0,\dots,0,1,0,\dots,0) $, where the only $1 $ is in the $i $-th place and $\hat{\varepsilon_i} $ which is a $k$-uple with all one's but the $i $-th place, which is filled by $0 $, i.e. $\widehat{\varepsilon_i}=(1, \dots, 1,0,1, \dots ,1) $. 
\end{notation}

\begin{definition} \label{secant}
The $r$-th secant variety  of $X $ is $\sigma_r(X):= \overline{\bigcup_{p_, \ldots, p_r\in X} \langle p_1 , \ldots , p_r}\rangle $
where the closure is the  the Zariski closure.  The set of points of $X$-rank equal to $r$ is sometime denoted as $\sigma_r^0(X)$. 
If  $\dim \sigma_r(X)< \min \{rn+r-1, \dim \langle X \rangle\}$, the variety $X$ is said to be \emph{$r$-defective}, otherwise $X$ is $r$-regular. If $X$ is $r$-defective, the difference $\delta_r= \min \{rn+r-1, \dim \langle X \rangle\}- \dim \sigma_r(X)$ is called the \emph{$r$-th secant defect of $X$}.
\end{definition}

We will often use the so called Concision/Autarky property (cf. \cite[Prop. 3.1.3.1]{Lands} \cite[Lemma 2.4]{bb3}) that we recall here.

\begin{lemma}[Concision/Autarky]\label{concision}
For any $q \in \mathbb{P}(V_1 \otimes \cdots \otimes V_k) $, there is a unique minimal multiprojective space $Y'\simeq \mathbb{P}^{n'_1}\times \cdots \times \mathbb{P}^{n'_k}  \subseteq Y\simeq \mathbb{P}^{n_1}\times \cdots \times \mathbb{P}^{n_k} $ with $n'_i \leq n_i$, $i=1 , \ldots, k$ such that $\mathcal{S}(Y,q)=\mathcal{S}(Y',q) $. 

\end{lemma}

\begin{definition}{(Concise Segre)} Given a point $q \in \mathbb{P}^N  $, we will call \emph{concise Segre} 
the variety $X_q:=\nu(Y')$ where $Y' \subseteq Y$ is the minimal multiprojective space $Y'\subseteq Y $ such that $q \in \langle\nu(Y') \rangle $ as in Concision/Autarky Lemma \ref{concision}.
\end{definition}

\begin{remark}\label{remarkconcision}
The minimal $ Y'$ defining the concise Segre of a point $q$ can be obtained as follows. Fix any $A \in \mathcal{S}(Y,q) $, set $A_i:= \pi_i(A) \subset \mathbb{P}^{n_i}$,  $i=1,\dots , k $, where the $\pi_i$'s are the projections on the $i$-th factor of Notation \ref{pi}. Each $\langle A_i \rangle \subseteq \mathbb{P}^{n_i} $ is a well-defined projective subspace of dimension at most $\min \{ n_i, r_X(q)-1 \} $. By Concision/Autarky we have $Y' = \prod_{i=1}^k \langle A_i \rangle$. 
 In particular $q $ does not depend on the $i$-th factor of $Y$ if and only if for one $A \in \mathcal{S}(Y,q) $ the set $\pi_i(A) $ is a single point.
\end{remark}

\begin{remark}\label{retta}
Let $ q \in \mathbb{P}^N$ and consider $A \in \mathcal{S}(Y,q) $. We claim that there is no line $L \subset X $ such that $\sharp(L \cap \nu(A) )\geq 2$. Obviously if $\sharp(L \cap \nu(A) )>2$ we would have at least $3 $ points that evince the rank of $q $ on a line, which is a contradiction with the linearly independence property that sets in $ \mathcal{S}(Y,q)$ have. So assume that there exists a line $L\subset X$ such that $\sharp(L \cap \nu(A) )=2$; let $u,v \in A $ be the preimages of those points, i.e. $u\neq v $ and $\langle \nu(u), \nu(v)  \rangle = L$. Then $r_X(q) > 2 $ because if $r_X(q)=2 $ then we would have $q \in L \subset X $, so the rank of $q $ will be $ 1$. Let $E= A\setminus \{ u,v \} $. Then we will have that $q \in \langle \nu (E) \cup L \rangle $, so we can find a point $o \in L $ such that $q \in \langle \nu (E) \cup \{ o \} \rangle $, which would imply $ r_X(q)< \sharp A $.
\end{remark}

\subsection{A very useful lemma}

Let $X$ be a non degenerate irreducible projective variety embedded in $\mathbb{P}^N $ via an ample line bundle $ \mathcal{L}$. Let $ Z\subset X$ be a zero-dimensional scheme and let $ D\subset \mathbb{P}^N$ be a fixed hyperplane, i.e. $D\in \vert \mathcal{L}\vert $.
Denote with $ Res_D(Z)$  the \emph{residual scheme of $ Z$} with respect to $ D$, i.e. the zero-dimensional scheme whose defining ideal sheaf is $ \mathcal{I}_{Z}:\mathcal{I}_D$.  The ideal sheaf $\mathcal{I}_{D\cap Z,D}\otimes \mathcal{L}$ represents the scheme theoretic intersection of $ D$ and $ Z$, also called \emph{the trace of $ Z$} with respect to $ D$.
The residual exact sequence of $ Z$ with respect to $ D$ in $ X$ is the following:
$$ 0 \rightarrow \mathcal{I}_{Res_D(Z)}\otimes \mathcal{L}(-D) \rightarrow \mathcal{I}_Z \otimes \mathcal{L} \rightarrow \mathcal{I}_{D\cap Z,D}\otimes \mathcal{L} \rightarrow 0.$$

An extremely useful tool that will turn out to be crucial in many proofs of this paper is  \cite[Lemma 5.1]{BalBer2}. We recall here the analogous statement given in  \cite[Lemma 2.4]{BalBerChristGes}   in the setting of zero-dimensional schemes.
\begin{lemma}[Ballico--Bernardi--Christandl--Gesmundo]\label{lemma: 5.1 plus}
 Let $X \subseteq \mathbb{P}^n$ be an irreducible variety embedded by the complete linear system associated with $\mathcal L = \mathcal O_X(1)$. Let $p\in \mathbb{P}^n$ and let $A,B$ be zero-dimensional schemes in $X$ such that $p \in \langle A \rangle$, $p \in \langle B \rangle$ and there are no $A' \subsetneq A$ and $B' \subsetneq B$ with $p \in \langle A'\rangle$ or $p \in \langle B' \rangle$. Suppose $h^1(\mathcal{I}_{B}(1)) = 0$. Let $C \subseteq \mathbb{P}^n$ be an effective Cartier divisor such that $\mathrm{Res}_C(A) \cap \mathrm{Res}_C(B) = \emptyset$. If $h^1 \big(X, \mathcal{I}_{\mathrm{Res_C(A \cup B)}} (1) (-C)\big) = 0$ then $A \cup B \subseteq C$.
\end{lemma}

We rephrase it in terms of sets of points of multiprojective spaces embedded via $ \vert \mathcal{O}(1,\dots,1) \vert $. \\
Let $ k\geq 2$, let $Y=\mathbb{P}^{n_1}\times \cdots \times \mathbb{P}^{n_k} $ such that $X:=\nu(Y)\subset \mathbb{P}^N $, where $N=\prod(n_i+1)-1 $. Let $ q \in \mathbb{P}^N$ be a point of $ X$--rank $r$ and let $ A,B\in \mathcal{S}(Y,q) $ be sets of points evincing the rank of $ q$ and write $S:=A\cup B $. In this setting, the irreducible variety $ X$ considered in Lemma \ref{lemma: 5.1 plus} is the Segre variety. The residual scheme $Res_C(S) $ is therefore $S\setminus (S\cap C)$. The assumption $h^1(\mathcal{I}_B(1))=0 $ of \cite[Lemma 2.4]{BalBerChristGes}, in the setting of Segre varieties becomes $h^1(\mathcal{I}_B(1,\dots,1))=0 $, which means that the points of $ \nu(B)$ are linearly independent and this assumption is satisfied since both $ A$ and $ B$ are sets evincing the rank of $q$.

With all this said we can state the specific version of \cite[Lemma 2.4]{BalBerChristGes} and \cite[Lemma 5.1]{BalBer2} which is needed in the present paper.

\begin{notation}
With an abuse of notation, when will will make cohomology computation, if the variety for which we compute the cohomology of the ideal sheaf is $Y$ we will omit it. We will specify the variety only when it is not $Y$.
\end{notation}

 \begin{lemma}\label{lemmariformulato}Let $ k\geq 2$ and consider $Y=\mathbb{P}^{n_1}\times \cdots \times \mathbb{P}^{n_k} $, where all $n_i\geq 1 $. Let  $q\in \mathbb{P}^N$, $A,B\in \mathcal{S}(Y,q) $ be two different subsets evincing the rank of $ q$ and write $ S=A\cup B$. Let $D \in \vert \mathcal{O}_Y(\varepsilon) \vert  $ be an effective Cartier divisor such that $A\cap B \subset D $, where $ \varepsilon=\sum_{i \in I} \varepsilon_i$ for some $ I\subset \{ 1,\dots, k\}$ as introduced in Notation \ref{epsiloneco}. If $h^1(\mathcal{I}_{S\setminus S\cap D}(\hat{\varepsilon}))=0 $ then $S\subset D$.
\end{lemma}
The above lemma gives a sufficient condition so that the whole $ S=A\cup B$ is contained in a given divisor $ D$ of the variety $ X$.
If $ A,B$ are two disjoint distinct sets evincing the rank of a tensor $q$ of $ X$--rank $ 3$ the assumption that $A\cap B \subset D $ is always satisfied.


\section{Identifiability on the 2-nd secant variety}\label{Sec:id:2}

In this section we study and completely determine the identifiability of points on the second secant variety of a Segre variety.

By Remark \ref{remarkconcision}, the concise Segre of a border rank-2 tensor $q$ is $X_q=\nu\left(\mathbb{P}^1_i\right)^{\times k}$.
Therefore for  the rest of this section we will  focus our attention to Segre varieties of products of $\mathbb{P}^1$'s.

\begin{remark}\label{2fattorisigma2} If the concise Segre $X_q$ of a tensor $q\in \sigma_2(X)$ is a $\nu( \mathbb{P}^{1}\times \mathbb{P}^{1}) $, then  $\sigma_2(X_q)$ parameterizes the $2 \times 2 $ matrices for which it  is trivial to see that they can be written as sum of two rank-$1 $ matrices in an infinite number of ways.
\end{remark}

For the rest of this section we will therefore focus on Segre varieties of $(\mathbb{P}^1)^{\times k}$ with $k\geq 3$. 

\begin{definition}
The variety $\tau(X)$ is the tangent developable of a projective variety $X$, i.e. $\tau(X)$ is  defined by the union of all tangent spaces to $X$.
\end{definition}

Recall that a tensor $q\in \tau(X) \setminus X $ has rank equal to $2$ if and only if the concise Segre $X_q$ of $q$ is a two-factors Segre, moreover it is not-identifiable for any number of factors (cf. eg. \cite[Remark 3]{BalBer3}).

\begin{proposition}\label{sigma20} Let $q \in \sigma_2^0(X) $. Then $| \mathcal{S}(Y,q)|> 1 $ if and only if the concise Segre $X_q$ of $q$ is  $X_q=\nu(\mathbb{P}^1 \times \mathbb{P}^1) $. 
\end{proposition}

\begin{proof}
We only need to check the case of $k \geq 4$ since $k=2,3$ are classically known. The case of matrix is obviously not-identifiable (cf. Remark \ref{2fattorisigma2}), while the identifiabily in the case $k=3$ is classically attributed to Segre and it is also among the so called Kruskal range (cf. \cite{Kruskal}, \cite[Thm. 4.6]{cov}, \cite[Thm. 1.2]{C.O.}), see also \cite[line $7$ of page $ 484$]{CMR}. We assume therefore that $k\geq 4 $.

Since $X$ is cut out by quadrics, then if a line $L \subset \mathbb{P}^N$ is such that $\deg (L \cap X)>2 $ then $L \subset X $ and the points of $L $ have $X$-rank $1$. 
Let $A,B \in\mathcal{S}(Y,q)$, either $\langle A \rangle =\langle B \rangle$ or $\langle A \rangle \cap\langle B \rangle=\{q\}$. In fact, in the first  case $A=B$ since $r_X(q)=2$ and therefore $\langle A\rangle$ is not contained in $X$, moreover $X$ is cut out by quadrics. In the second case $A\neq B$. 
We can therefore assume that
$A,B \in \mathcal{S}(Y,q) $ are two disjoint sets:
$A=\{a, a' \} $, $B=\{ b, b' \} $, where $a=(a_1, \dots, a_k) $, $a'=(a_1',\dots, a_k') $ and $b=(b_1,\dots,b_k) $, $b'=(b_1',\dots,b_k') $. Since $a \neq a' $, we may assume that at least one of their coordinates is different. 
Actually we can  assume that all the $a_i\neq a_i'$, otherwise,  by the concision property, one could consider one factor less. The same considerations hold for $B $. 

Now suppose 
that there exists an index $i\in \{1 ,\dots,k \} $ such that $\{a_i,a_i' \}  \neq  \{ b_i, b_i'\} $ and  let such an index be $i=1 $: 
$\{a_1,a_1' \} \neq \{ b_1, b_1'\} $. 
\\
Now we proceed by induction on $k$.
Let $\eta_k$, $\nu_k $, and $X_k $ be as in Notation \ref{pi}.
Let $\tilde{q}=(q_1,\dots,q_{k-1}) $ be the projection $ \eta_k(q)
$, then $\eta_k(A)\neq \eta_k(B) $ and $ \emptyset \neq \langle \nu_k(\eta_k(A)) \rangle  \cap \langle \nu_k(\eta_k(B)) \rangle \supset \{ \tilde{q} \} $ because $\{ q\}\subset \langle \nu(A) \rangle \cap \langle \nu(B) \rangle $. So $r_{X_k}(\tilde{q})= 2$ and $| \mathcal{S}(Y_k,\tilde{q})|\geq 2 $, which is a contradiction because $X_k $ is a concise Segre of $k-1 $ factor (where $ k>3$) and a point of it cannot have more than a decomposition. Thus for all $i=, 1\ldots , k $ we have that $\{ a_i,a'_i\}=\{b_i,b'_i \} $.

Without loss of generality assume that $a_1=b_1 $ and $a'_1=b'_1 $, moreover up to permutation there exists an index $e \in \{1, \dots , k-1 \} $ such that $b_i=a_i $ and consequently $b'_i=a'_i $ for $1 \leq i \leq  e $ and $b_i=a'_i $ and $b'_i=a_i $ for $e+1 \leq i \leq k $. 
Eventually by exchanging the role of the first $e $ elements with the others, we have that $k-e \geq 2 $ because by assumption $k \geq 4 $. Let $H \in |\mathcal{O}_Y(0,\dots,0,1) | $ be the only element containing $ a'$, $H= \mathbb{P}^1 \times \cdots \times \mathbb{P}^1 \times \{ a'_k\} \cong (\mathbb{P}^1)^{\times k-1}$; then $\operatorname{Res}_H(A\cup B)=\{ a',b'\} $ and since $k-e\geq 2 $ we have that $\eta_k(a')\neq \eta_k(b') $, i.e. $h^1(\mathcal{I}_{\operatorname{Res}_H(A \cup B)}(1,\dots,1,0))=0 $. By  Lemma \ref{lemmariformulato},  
we get $a'=b' $ which contradicts the fact that $A \cap B=\emptyset $.
\end{proof}

\begin{corollary}\label{spaziosolS2} Let $q$ be any rank-2 tensor.
If $q $ is not-identifiable, then there is a bijection between $ \mathcal{S}(Y,q)$ and $\mathbb{P}^2 \setminus L $, where $L\subset \mathbb{P}^2 $ is a projective line, $q \in \tau (X) $ and $ L$ parametrizes the set of all degree $2 $ connected subschemes $V $ of $Y$ such that $q \in \langle \nu (V) \rangle  $. 
\end{corollary}

\begin{proof}
It suffices to work with a Segre variety of 2 factors only because by Proposition \ref{sigma20} it is the unique not-identifiable case in rank-2. Thus $X\subset \mathbb{P}^3 $ is a quadric surface. Denote by $ H_q \subset \mathbb{P}^3$ the polar plane of $X $ with respect to $q $. Since $q \notin X $ we have that $q \notin H_q $ and the intersection $X \cap H_q =\{ p \in X \; |\; T_pX \ni q \}$ is a smooth conic. 
Remark also that by definition a point $o \in X $ is such that $q \in T_oX $ if and only if $o \in X \cap H_q \subset  \tau(X) $. \\
Fix $o \in H_q $, then
\begin{itemize}
\item if $o \notin X $, the line given by $ \langle o,q \rangle $ is not tangent to $X $ and when considering the intersection $\langle o,q\rangle \cap X $, it is given by two points $p_1, p_2 \notin \{o,q\} $ such that $\{ p_1, p_2\} \in \mathcal{S}(Y,q) $; 
\item if $o \in X $, i.e. $o \in X \cap H_q $, then the line $ \langle o,q \rangle $ is tangent to $X $.
\end{itemize}
Consider $\Pi_q= \{ \textnormal{lines } L \subset \mathbb{P}^3 \textnormal{ passing through } q\} \cong \mathbb{P}^2 $ and consider the following isomorphism $\varphi: H_q \longrightarrow \Pi_q $ defined by $p \mapsto \langle p,q \rangle $. 
Clearly $\varphi (X \cap H_q) $ is  a smooth conic $ \mathcal{C}$ of $\Pi_q $. Moreover one can notice that $\Pi_q \setminus \varphi (X \cap H_q) \cong \mathbb{P}^2 \setminus \mathcal{C} $ are just the points of the first case.
\end{proof}

\section{Examples of not-identifiable rank-3 tensors}\label{sec.ex}
The purpose of this section is to explain in details the phenomena behind the not-identifiable rank-3 tensors. In the main Theorem  \ref{main_theorem} they will turn out to be the unique cases of not-identifiability for a rank-3 tensor.

From now on we always consider $q \in \mathbb{P}^N $ such that $r_X(q)=3 $, 
therefore, by Remark \ref{remarkconcision}, 
we may assume that $q$ is an order-$k$ tensor with at most 3 entries in each mode, i.e. the concise Segre of $q$ is $X_q= \nu(\mathbb{P}^{n_1} \times \cdots \times  \mathbb{P}^{n_k} )$, with $n_1, \dots ,n_k \in \{1,2 \}$. 

First of all let us remark that the matrix case is highly not-identifiable even for the rank-3 case.

\begin{remark}\label{caso1} 
In the case of two factors (i.e. $k=2$), a rank-3 tensor $q$ is a $3\times 3 $ matrix  of full rank.
The dimension of the concise Segre $X$ of $ 3\times 3$ matrices is $4 $ and 
$\dim(\sigma_3(X))=\min \{\dim(\mathbb{P}^8), 3\dim(X)+2 \}=\min\{8,14\}=8$. Thus $\dim \mathcal{S}(Y,q)= 14-8=6 $ for all $ q \in \mathbb{P}^8$ of rank $ 3$.
\end{remark}

Consider now the third secant variety of the Segre embedding of $Y= \mathbb{P}^{n_1} \times \cdots \times \mathbb{P}^{n_k} $, where 
$n_i\in \{1,2 \}$, the following Examples 
 \ref{caso3} and \ref{caso4} and
 Proposition \ref{x1.1} provide instances of not-identifiability that we will show to be essentially the only classes of not-identifiable rank-3  tensors in $\mathbb{C}^{n_1+1}\otimes \cdots \otimes \mathbb{C}^{n_k+1}$ (cases \eqref{3.}, \eqref{4.} and \eqref{5.} respectively of our main Theorem \ref{main_theorem}) more than the well known ones (matrix case, points on tangential variety of $\nu((\mathbb{P}^1)^{\times 3})$, and elements of the defective $\sigma_3(\nu((\mathbb{P}^1)^{\times 4}))$ -- items \eqref{1.}, \eqref{2.} and \eqref{6.} respectively of Theorem \ref{main_theorem}). 

In the following remark we explain the behaviour on $\sigma_3((\mathbb{P}^{1})^{\times 4})$.

\begin{remark}\label{allP14}
It has been shown in \cite{AOP} (cf. also \cite{CGGrank, CGG}) that  the third secant variety of a Segre variety $X$ is never defective unless
either $X=\nu(\mathbb{P}^1 \times \mathbb{P}^1 \times \mathbb{P}^1 \times \mathbb{P}^1) $ or $X=\nu(\mathbb{P}^1 \times \mathbb{P}^1 \times \mathbb{P}^a) $, with $a\geq 3 $.


The case in which $q$ is a rank-3 tensor in $\langle \nu(\mathbb{P}^1 \times \mathbb{P}^1 \times \mathbb{P}^a) \rangle$ with $a\geq 3$ corresponds to a not-concise tensor (cf. Remark \ref{remarkconcision}) therefore it won't  play a role in our further discussion.

The case in which $X=\nu(\mathbb{P}^1 \times \mathbb{P}^1 \times \mathbb{P}^1 \times \mathbb{P}^1)$ and  $q\in \langle X \rangle$ can also be easily handled. The fact that  $\dim(\sigma_3(X))$ is strictly smaller than the expected dimension proves that the generic element of $\sigma_3(X)$ has an infinite number of rank-3 decompositions. By definition of dimension there is no element of $\sigma_3(X)$  s.t. its tangent space has dimension equal to the expected one: $\dim(T_q(\sigma_3(X)))\leq \dim \sigma_3(X)$ for all $q\in \sigma_3(X)$. This does not exclude the existence of certain special rank-3 tensors $q$ such that $\dim(T_q(\sigma_3(X)))=\dim(T_{q'}(AbSec_3(X)))< 14$ where $AbSec_3(X):=\{(p; p_1, p_2,p_3)\in \mathbb{P}^{15} \times X^{\times 3} \, | \, p \in \langle p_1, p_2, p_3 \rangle\}$ is the 3-th abstract secant of $X$ and $q'$ is the preimage of $q$ via the projection on the first factor. The impossibility of the existence of such a point is guaranteed by \cite[Cap II, Ex 3.22, part (b)]{Hart}. This proves that all the tensors of $\sigma_3^0(X)$ have an infinite number of rank-3 decompositions.
\end{remark}

Before explaining the other not-identifiable examples we need some preliminary results.

\begin{remark}\label{sulP2aebindip}
Let $Y $ be a multiprojective space with at least two factors where at least one of them is of projective dimension $2 $. By relabeling, if necessary, we can assume that the first factor is a $\mathbb{P}^2 $. Let $q\in \sigma_3^0(\nu(Y))$, with $\nu(Y)$ the concise Segre of $q$ 
and let $A,B \in \mathcal{S}(Y,q)$ be two disjoint subsets evincing the rank of $q $. By Autarky $\langle \pi_1(A)\rangle = \langle \pi_1(B)\rangle = \mathbb{P}^2 $;
  moreover when considering the restrictions of the projections $\pi_{1\vert A} $ and $\pi_{1\vert B} $ to the subsets $A $ and $B $ respectively, they are both injective and both $\pi_1(A) $ and $\pi_1(B) $ contain linearly independent points.
\end{remark}

\begin{remark}
\label{dimensione caso2} 
Consider  $Y= \mathbb{P}^2 \times \mathbb{P}^1 \times \mathbb{P}^1 $  and an irreducible divisor $G \in |\mathcal{O}_Y(0,1,1) | $. Then $\sigma_2(\nu(G))\subsetneq \sigma_3(\nu (G)) =\langle \nu (G) \rangle = \mathbb{P}^8$.
Indeed $G$ is nothing else than the Segre-Veronese variety (\cite{b}) of $\mathbb{P}^2\times \mathbb{P}^1$ embedded in bi--degree (1,2), i.e. $G\cong \mathbb{P}^2 \times \mathbb{P}^1 $, $ \mathcal{O}_Y(1,1,1)_{\rvert _{G}}\cong \mathcal{O}_{\mathbb{P}^2 \times \mathbb{P}^1}(1,2)$ and $\mathcal{O}_Y(1,0,0)\cong \mathcal{O}_Y(1,1,1)(-G)$. The classification of the dimensions of secant varieties of such a Segre-Veronese can be found in 
\cite{BD,ChCi,BCC, BBC}.
\end{remark}

\begin{proposition}\label{dimensione caso3}
For the Segre embedding of $Y=\mathbb{P}^2 \times \mathbb{P}^1 \times \mathbb{P}^1 $ fix $G_1 \in |\mathcal{O}_Y(0,1,0) | $ and $G_2 \in |\mathcal{O}_Y(0,0,1) | $ and define $G:= G_1 \cup G_2$ to be their union. 
We have that for $\{i,j\}=\{1,2\}$, $ \dim\langle \nu(G_i)\rangle=5 $, $\dim\langle \nu(G)\rangle=8$, $\sigma_2(\nu(G_i))=\langle \nu (G_i) \rangle $ and $ \langle \nu (G) \rangle$ is the join of $\sigma_2(\nu(G_i)) $ and $\nu (G_j) $.
\end{proposition}

\begin{proof} 
 First of all remark that, for $i=1,2$, $G_i\cong \mathbb{P}^2 \times \mathbb{P}^1 $,  $\mathcal{O}_Y(1,1,1)_{\rvert _{G_i}} \cong \mathcal{O}_{\mathbb{P}^2\times \mathbb{P}^1}(1,1) $ and $G$ is a reducible element of $|\mathcal{O}_Y(0,1,1) | $. 
  With an analogous computation of the one in Remark \ref{dimensione caso2} one sees that $\dim \langle \nu(G) \rangle=8 $ and 
  $\sigma_2(\nu(G_i))=\langle \nu (G_i) \rangle $. 
It remains to show that $ \langle \nu (G) \rangle= \mathcal{J}$, where $\mathcal{J}$ denotes the join of $\sigma_2(\nu(G_i)) $ and $\nu (G_j) $ with $\{i,j\}=\{1,2\}$.
We remark that since $\sigma_2 (\nu(G))= \mathbb{P}^5$, then $\mathcal{J}=\mathrm{Join}(\mathbb{P}^3,\nu(G_1),\nu(G_2))  $. In order to show that $\mathcal{J}=\mathbb{P}^8$ it is sufficient to see that $\dim(\sigma_2(\nu(G_i)\cap \nu(G_j)))=1$ and this is a straightforward computation since the elements of $\nu(G_1)$ are tensors with a second factor fixed, while the elements of $\nu(G_2)$ have the third factor fixed, and in order to have the equality between an element of $\sigma_2(\nu(G_1))$ and an element of $\nu(G_2)$ it is sufficient to impose two linear independent conditions and therefore since $\dim(\nu(G_2))=3$ we have that the intersection has dimension 1.
\end{proof}

\begin{example}\label{caso3}
Take $Y=\mathbb{P}^2 \times \mathbb{P}^1 \times \mathbb{P}^1 $, consider the Segre embedding on the last two factors and take a hyperplane section which intersects $\nu(\mathbb{P}^1 \times \mathbb{P}^1) $ in a conic
$\mathcal{C }$, then take a point $q \in \langle \nu (\mathbb{P}^2 \times \mathcal{C} )\rangle $. Such a construction is equivalent to consider an irreducible divisor $G \in |\mathcal{O}_Y(0,1,1) | $, so $ G\cong \mathbb{P}^2 \times \mathbb{P}^1$ embedded via $\mathcal{O}(1,2)$, then $\dim \sigma_2(\nu(G))=7 $, thus $\sigma_2(\nu(G)) \subsetneq \langle \nu(G) \rangle \simeq \mathbb{P}^8 $.  
As a direct consequence we get that a general point $q \in \langle \nu(G) \rangle $ has $\nu(G) $-rank $3 $ and it is not-identifiable because of the not-identifiability of the points on $\langle \mathcal{C} \rangle $ and by \cite[Cap II, Ex 3.22, part (b)]{Hart}.  
Thus $ \dim(\mathcal{S}(G,q))= 3 $. 
\end{example}

 The following example is in the same setting of the previous one, but in this case we deal with a reducible conic and in such a case we get a $4 $-dimensional family of solutions.

\begin{example}\label{caso4} 
Fix $Y=\mathbb{P}^2 \times \mathbb{P}^1 \times \mathbb{P}^1$. Consider $ G_1 \in |\mathcal{O}_Y(0,0,1)|$, $ G_2 \in |\mathcal{O}_Y(0,1,0)|$ and call $G=G_1 \cup G_2$ which is a reducible element of $ |\mathcal{O}_Y(0,1,1)|$. By Proposition \ref{dimensione caso3},  $\dim \langle \nu(G) \rangle=8 $, moreover by a dimension count we have $\langle \nu(G_i)\rangle= \sigma_2(G_i) $, for $i=1,2 $, both having dimension $5 $. 
By Proposition \ref{dimensione caso3} we also have that $\langle \nu(G) \rangle=\mathcal{J}_1=\mathcal{J}_2$, where $\mathcal{J}_1 =\textnormal{Join}( \sigma_2(\nu(G_1)), \nu(G_2))$ and $\mathcal{J}_2= \textnormal{Join}( \sigma_2(\nu(G_2)) , \nu(G_1))$. 
A general $q \in \langle \nu(G)\rangle $ has rank $3$ and for the subsets evincing its rank we have a $4 $-dimensional family of sets $A $ such that $\sharp(A)=3$, $\sharp(A\cap G_1)=2$, $\sharp (A \cap G_2)=1 $, $A\cap G_1 \cap G_2=\emptyset $ and $q \in \langle \nu(A)\rangle $. Such a family has dimension 4 since $G_1$ is a non defective threefold in $\mathbb{P}^5$, therefore there exists a $2$-dimensional family of sets of cardinality $2$ in $G_1$ spanning a general point of $\mathbb{P}^5$; moreover $q$ sits in a $2$-dimensional family of lines joining points of $G_1$ and $G_2$. 
Analogously, by looking at $q $ as an element of $\mathcal{J}_2 $, we get the existence of a $4 $-dimensional family of sets $B $ such that $\sharp(B)=3$, $\sharp(B\cap G_2)=2$, $\sharp (B \cap G_1)=1,$ $ A\cap G_1 \cap G_2=\emptyset $ and $q \in \langle \nu(B) \rangle $. So we proved that $\mathcal{S}(G,q)$ contains at least two dimensional families of solution. Thus $\dim\mathcal{S}(G,q)\geq 4 $. 
\end{example}

\begin{proposition}\label{nonidesempi} 
Let $q\in \sigma_3^0(\nu(\mathbb{P}^{2}\times \mathbb{P}^{1}\times \mathbb{P}^{1}))$ and suppose that there exist $A,B\in \mathcal{S}(Y,q)$ s.t. $\sharp (A\cup B)=6$. Then there exist a unique $G\in |\mathcal{O}_Y(0,1,1)|$ containing $S=A\cup B$. For such a $G$ we have that $\mathcal{S}(Y,q)=\mathcal{S}(G,q)$.
\end{proposition}


\begin{proof} Call $S:=A\cup B $, by Remark \ref{sulP2aebindip}, both $\pi_{1\vert A} $ and $\pi_{1\vert B} $ are injective and both $\pi_1(A) $ and $ \pi_1(B)$ are sets containing linearly independent points. So $ h^1(\mathcal{I}_A(1,0,0))=h^1(\mathcal{I}_B(1,0,0))=0$.  
Now $h^0(\mathcal{O}_Y(0,1,1))=
4$, so there exists $G \in \vert \mathcal{O}_Y(0,1,1) \vert $ containing $B $. Moreover $ S\setminus S\cap G\subseteq A$ but since $ h^1(\mathcal{I}_{A}(1,0,0))=0$ we have that $ S\subset G$. This holds for any $G \in \vert \mathcal{I}_B(0,1,1)\vert $, so $\langle \nu_1(\eta_1(A))\rangle \subset \langle \nu_1(\eta_1(B))\rangle $. The same holds exchanging the roles of $A $ and $B $, thus $ \langle \nu_1(\eta_1(A))\rangle = \langle \nu_1(\eta_1(B))\rangle $.

Assume $G $ is irreducible, then $B $ contains three linearly independent points on $G $, so the points of $ B$ are uniquely determined by $ G$. 

Assume $G $ is reducible, i.e. $G=G_1 \cup G_2 $, with $G_1 \in \vert \mathcal{O}_Y(0,1,0) \vert $ and $G_2 \in \vert \mathcal{O}_Y(0,0,1) \vert $. Remark that, by Autarky, it does not exist any $E \in \mathcal{S}(Y,q) $ which is all contained in $G_i$, for $i=1,2$, because $ G$ is a multiprojective subspace of $ Y$.
Without loss of generality, we may assume that two points of $E$ lies in $G_1 $; then the three points of $ E$ are uniquely determined by a reducible conic, i.e. by the reducible element $G=G_1 \cup G_2 $  that contains them.
\end{proof}

\begin{corollary}\label{P2P1P1}
If $q\in \sigma_3^0(\nu(\mathbb{P}^{2}\times \mathbb{P}^{1}\times \mathbb{P}^{1}))$ is such that there exist two disjoint sets  $A, B\in \mathcal{S}(Y,q)$, then $q$ can be either as in Example \ref{caso3} and  $ \dim(\mathcal{S}(Y,q))= 3 $ or as in Example \ref{caso4} and  $ \dim(\mathcal{S}(Y,q))= 4 $.
\end{corollary}

\begin{proof}
This is a direct consequence of the uniqueness of the $G \in |\mathcal{O}_Y(0,1,1)|$ s.t. $\mathcal{S}(Y,q)=\mathcal{S}(G,q)$ in Proposition \ref{nonidesempi}.
\end{proof}

\begin{proposition}\label{x1.1}
Let $ Y':=\mathbb{P}^1\times \mathbb{P}^1\times \{ u_3\} \times \cdots \times \{ u_k\}$ be a proper subset of $Y=\mathbb{P}^{n_1}\times \cdots \times \mathbb{P}^{n_k} $, $k\geq 2 $. Take $q' \in \langle \nu(Y') \rangle\setminus \nu(Y') $, $A\in \mathcal{S}(Y',q') $ and $p\in Y \setminus Y' $. Assume that $ Y$ is the minimal multiprojective space containing $ A\cup \{ p\}$ and take $q \in \langle  \{q',\nu(p) \}  \rangle \setminus \{q' , \nu(p) \}  $. 
\begin{enumerate}
    \item\label{f1} $\sum_{i=1}^k n_i\geq3$,  $n_1,n_2\leq 2$, $n_3, \ldots , n_k\leq 1$ and if $k\geq 3 $ then $ r_{\nu (Y)}(q)>1$;
    \item\label{f2} If $k\geq 3$ and $\sum_{i=1}^k n_i\geq 4$ then $r_{\nu(Y)}(q)=3$ and $\mathcal{S}(Y,q)=\{ \{p\} \cup A\}_{A\in \mathcal{S}(Y',q')}$.
  
    \item\label{f3} $\nu(Y)$ is the concise Segre of $ q$.
    
\end{enumerate}

\end{proposition}

\begin{proof}

First of all remark that $r_{\nu(Y)}(q)>1,$ otherwise there exists $o\in Y$ s.t.  $q =\nu (o)$ and $q' \in \langle \nu (\{o,p\})\rangle$. Since $r_{\nu(Y)}(q') =2$,
we would have $\{o,p\}\in \mathcal{S} (Y,q')$ and by Autarky we get $\{o,p\}\subset Y'$, contradicting the assumption $p\notin Y'$.
\\
The fact that $ n_1+\cdots +n_k\geq 3$ is obvious from the fact that $p \notin Y' $ so $ Y\neq Y'$.
\\
Since $q' $ is a $2 \times 2 $ matrix of rank $2$,  $\dim \mathcal{S}(Y',q')=2 $ and $Y' $ is the minimal multiprojective subspace of $Y $ containing $ A$,  the minimal multiprojective subspace containing $Y'\cup \{ p\} $ is $Y $. So since $\mathbb{P}^{n_i}=\langle \pi_i(Y'\cup \{ p \} )\rangle  $, we get 
$1\leq n_i \leq 2$ for  $i=1,2$ and $n_i=1$ for all $i>2$. 
This ends item \ref{f1}.

Item  \ref{f3} will be 
a consequence of item \ref{f2}, in fact if the structure of the elements on $\mathcal{S}(Y,q)$ is of type $A\cup \{p\}$ with $A\in \mathcal{S}(Y',q')$, then Autarky and the fact that $Y$ is the minimal multiprojective space containing $A\cup \{p\}$ will imply that $\nu (Y)$ is the concise Segre of $q$. So let us prove item \ref{f2}.
 
The proof is by induction on the number of factors. Step \ref{esc1} is the basis of induction for the case in which $Y$ has at least one factor of projective dimension 2 ($k=3$), Step \ref{esc2} is the basis of induction for the case in which all the factors of $Y$ have projective dimension 1 ($k=4$), Steps \ref{esc2induzione} and \ref{esc1induzione} are the induction processes of Step \ref{esc2} and Step \ref{esc1} respectively.\\ 
Let $E\in \mathcal{S}(Y,q)$, if we will show that $E\supset \{ p \} $ and that there exists $B \in \mathcal{S}(Y',q') $, such that $E=B\cup \{p \} $, we will be done. Assume that there is no $ B\in \mathcal{S}(Y',q') $ such that $E=B\cup \{p \} $. Fix any $A\in \mathcal{S}(Y',q') $ and set $S:=A\cup \{ p \} \cup E $.

\begin{enumerate}[leftmargin=+.2in, label=(\Alph*)]

    \item \label{esc1} {[Case $k=3$, ${n_1=2}$, ${n_2=n_3=1}$]} First assume $p\in E$ and set $E':= E\setminus \{p\}$ and $F= A\cup E'$. 
Since $\cap _{B\in \mathcal{S} (Y,q')}\eta _3(B)=\emptyset$, taking another $A\in
\mathcal{S} (Y,q')$ if necessary we may assume $\eta _3(A)\cap \eta _3(E') =\emptyset$.  Set $\{D\}:= |\mathcal{I} _p(0,0,1)|$. By Lemma \ref{lemmariformulato}, we have $h^1(\mathcal{I} _{S\setminus S\cap D}(1,1,0)) >0$ and
hence (since $\sharp F\le 4$)  $h^0(\mathcal{I} _{S\setminus S\cap D}(1,1,0))\ge 3$. This must be true for all $A\in \mathcal{S} (Y',q')$ and
hence we have $h^0(Y_3,\mathcal{I} _{\eta _3(Y')\cup \eta _3(E')}(1,1)) \ge 3$. Since $\eta_3(Y')\in \vert \mathcal{O}_{Y_3}(1,1) \vert $ we have $h^0(Y_3,\mathcal{I}_{\eta_3(Y')}(1,1)) =1$, contradicting the previous inequality.

From now on suppose $p\notin E$.
As above we may assume $\eta _3(A)\cap \eta _3(E) =\emptyset$.

Fix $o\in E$. Since $h^0(\mathcal{O} _Y(1,1,0)) =6$ and $\sharp A \cup \{ p \} \cup \{ o\}= 4$ there is $G\in |\mathcal{O} _Y(1,1,0)|$ containing $A\cup \{p\} \cup
\{o\}$. Assume for the moment $S\nsubseteq G$, i.e. $E\nsubseteq G$. We have $h^1(\mathcal{I} _{S\setminus S\cap G}(0,0,1)) >0$, thus  $\sharp E\ge 3$. Since $\sharp E \le 3$, we get $\sharp E =3$ (and hence $q$ has rank $3$ and $\nu (Y)$ is the concise Segre containing $q$), $S\setminus S\cap G =E\setminus
\{o\}$ and $\sharp \pi _3(E\setminus \{o\}) =1$. Taking a different $o\in E$ we get $\sharp \pi _3(E) =1$, i.e. $\nu (Y)$ is
not the concise Segre of $q$, a contradiction.

Now assume $S\subset G$.   Since this must be true for
all
$G\in |\mathcal{I} _{A\cup
\{p,o\}}(1,1,0)|$, we get
$|
\mathcal{I} _{A\cup
\{p,o\}}(1,1,0)| \supseteq |\mathcal{I} _{\{p\}\cup E}(1,1,0)| \ne \emptyset$. Note that $\eta _3(Y')\in \vert \mathcal{O} _{Y_3}(1,1)\vert $ and hence
$h^0(Y_3,\mathcal{I} _{\eta _3(Y')}(1,1)) =1$. Since $n_1=2$ and $Y$ is the minimal multiprojective space containing $q$, we have
$\eta _3( p)\notin \eta _3(Y')$. Thus $h^0(Y_3,\mathcal{I} _{\eta _3(Y')\cup \{\eta _3( p)\}}(1,1)) =0$, a contradiction since $|\mathcal{I} _{A\cup\{p,o\}}(1,1,0)|\neq \emptyset$.

\item \label{esc2}  {[Case $k=4$, ${n_1=n_2=n_3=n_4=1}$]}
 Fix $G\in |\mathcal{O} _Y(0,0,1,1)|$ containing $E$. Assume $S\nsubseteq G$. Since $S\setminus E=A\cup \{ p\} $, by Lemma \ref{lemmariformulato}, we have $h^1(\mathcal{I} _{A \cup \{ p\} }(1,1,0,0)) >0$. Call $p'$ the projection of $p$ 
via $Y\to Y'$. Since $\mathcal{O} _{\mathbb{P}^1\times \mathbb{P}^1}(1,1)$ is very ample we get that either $p'\in A$ or that $\sharp (\pi
_i(A\cup \{p'\})) =1$ for some $i\in \{1,2\}$. The second possibility is excluded, because $\sharp (\pi _1(A))=\sharp (\pi
_2(A)) =2$ for any $A\in \mathcal{S} (Y',q')$. The first possibility is excluded taking instead of $A$ another general $A_1\in \mathcal{S}
(Y',q')$. Now assume $S\subset G$. We get $A\subset G$. 
This is ruled out taking another $A\in \mathcal{S} (Y',q')$ since a general
$a\in Y'$ is contained in some $B\in \mathcal{S} (Y',q')$. Thus we would have that $ Y'\subset G$ which is a contradiction.


\item \label{esc2induzione} {[Case $k\ge 5$, ${n_i=1}$ for all $i$'s]} We exclude this case by induction on $k$, the base case $k=4$ being excluded in \ref{esc2}. Fix $o\in \mathbb{P}^1\setminus \{p_k,u_k\}$, set
$M:= \pi _k^{-1}(o)$,  i.e. $M=(\mathbb{P}^1)^{\times k-1}\times \{ o\} $ and call $\Lambda := \langle \nu (M)\rangle$. Note that $(Y'\cup \{p\})\cap M=\emptyset$.  
Denote by $r=2^k-1$ and define $r':= \dim \Lambda = 2^{k-1}-1$.

Consider the following linear projection form $\Lambda $:
\begin{equation}\label{proj:l}
    \ell : \mathbb{P}^r\setminus \Lambda \to \mathbb{P}^{r'}.
    \end{equation}
Note that $\nu (Y) \cap \Lambda = \nu (Y_k)\times \{o\}$ and that $\ell _{|\nu (Y)\setminus M} =$ $\nu _k(\eta_k(Y\setminus M))$. We identify $\mathbb{P}^{r'}$ with the target projective space of $Y_k$. Since $(Y'\cup \{p\})\cap M = \emptyset$, $\ell$ is well-defined on $Y'\cup\{p \} $ and it acts as the composition of $\eta _k$ and the Segre embedding.

By the inductive assumption $\mathcal{S} (Y_k,\ell (q)) =\{ B \cup \eta _k({p}) \}_{B\in \mathcal{S} (\eta_k(Y'),\eta_k(q'))}$. Thus for any $E\in \mathcal{S} (Y,q)$ there is $B\in \mathcal{S}(Y',q')$
such that $\eta_k(E) =\eta_k(B\cup \{p\})$.  Since $\eta_{k|E}$ is injective by Remark  \ref{retta} and $\mathcal{S} (Y,q)\supseteq \{B\cup \{p\}\}_{B\in \mathcal{S} (Y',q')}$, we get $\mathcal{S} (Y,q) =\{B\cup \{p\}\}_{B\in \mathcal{S} (Y',q')}$.


\item \label{esc1induzione} {[Case $k\ge 3$, ${n_1=2}$, $n_1+\cdots +n_k \ge 5$]} 
If only one of the factors is a $\mathbb{P}^2 $ we use Step \ref{esc1} as base of the induction and then we construct a projection similar to \eqref{proj:l}. Indeed $Y=\mathbb{P}^2\times (\mathbb{P}^1)^{k-1} $, where $k\geq 4 $. Fix $o \in \mathbb{P}^1 \setminus \{ p_k, u_k\} $, set $M:=\pi_k^{-1}(o) $ and define $ \Lambda:= \langle \nu(M)\rangle$. Denote $ r=3\cdot 2^{k-1}-1$ and $r'=\dim \Lambda:=3\cdot 2^{k-2}-1 $. We consider the linear projection $\ell: \mathbb{P}^r \setminus \Lambda \to \mathbb{P}^{r'} $ which acts as the composition of $ \eta_k$ and the Segre embedding. By the inductive assumption $\mathcal{S} (Y_k,\ell (q)) =\{ B \cup \eta _k({p}) \}_{B\in \mathcal{S} (\eta_k(Y'),\eta_k(q'))}$.
Thus for any $E\in \mathcal{S} (Y,q)$ there is $B\in \mathcal{S}(Y',q')$
such that $\eta_k(E) =\eta_k(B\cup \{p\})$.  Since $\eta_{k|E}$ is injective by Remark  \ref{retta} and $\mathcal{S} (Y,q)\supseteq \{B\cup \{p\}\}_{B\in \mathcal{S} (Y',q')}$, we get $\mathcal{S} (Y,q) =\{B\cup \{p\}\}_{B\in \mathcal{S} (Y',q')}$.\\
Now assume also $n_2=2 $, so that we must have $k\geq 3 $. 
Let $Y=\mathbb{P}^2\times \mathbb{P}^2\times (\mathbb{P}^1)^{k-2} $ and fix $o\in \mathbb{P}^2\setminus \pi _2(Y')$. Set $M:= \pi _2^{-1}(o)$,  
and $\Lambda:= \langle \nu (M)\rangle$. Then $r=9\cdot 2^{k-2}-1$, $\dim \Lambda =9\cdot 2^{k-3}-1$. Let $ r':=9\cdot 2^{k-3}-1$ and consider the linear projection
 $\ell : \mathbb{P}^r\setminus \Lambda \to \mathbb{P}^{r'}$ from $\Lambda$ which acts on $\nu (Y)$ as the composition of the Segre embedding and the map $\mathbb{P}^2\times \mathbb{P}^2\times (\mathbb{P}^1)^{k-2}\setminus \mathbb{P}^2\times \{o\}\times (\mathbb{P}^1)^{k-2} \to \mathbb{P}^2\times (\mathbb{P}^1)^{k-1}$, which is the linear projection $\mathbb{P}^2\setminus \{o\}\to \mathbb{P}^1$ on the second factor and the identity on any other factor. Since $(Y'\cup \{p\})\cap M =\emptyset$, $\ell (q)$ is well-defined. We conclude since we already proved the statement in the case where only one of the factors is a $ \mathbb{P}^2$.

\end{enumerate}\vspace{-0,8cm}\end{proof}

\section{Lemmas}
In this section we collect the basic lemmas that we will need all along the proof of the main theorem of the present paper, Theorem \ref{main_theorem}.

\medskip 

The following two lemmas describe two very basic properties that two different sets $A$ and $B$ evincing the rank of the same rank-3 point $q$ have to satisfy.

\begin{lemma} \label{dimS2ranknot3}
 Let $q$ be a not-identifiable tensor  and let $A$ and $B$ two distinct sets evincing the rank of $q$. Define $S:=A \cup B $.
If $\sharp (S) \geq 5 $ and $\dim\langle \nu(S) \rangle =2 $, then the rank of $q$ cannot be 3.
\end{lemma}

\begin{proof}
Assume the existence of such a rank-3 tensor  $q $ with 2 distinct decompositions $A$ and $B$ s.t. $\sharp (A \cup B) \geq 5 $. 
The plane $\langle \nu(S) \rangle $  contains at least five not-collinear points. Note that $\langle \nu(S) \rangle \not\subseteq X $, otherwise also $q\in X $ which contradicts $r_X(q)=3 $. So $ \langle \nu(S) \rangle \cap X $ contains a conic $\mathcal{C}$. Either if it is reduced or not, the two secant variety of $ \mathcal{C}$ fills $\langle \nu(S) \rangle=\mathbb{P}^2$. So $r_X(q) \leq 2 $ , which is an absurd.
\end{proof}

\begin{lemma}\label{cardinalitatra5e6} 
Let $q$ be a not-identifiable rank-3 tensor and let $A,B \in \mathcal{S}(Y,q)$ be distinct. Then $\sharp (A \cap B) \leq 1 $.
\end{lemma}

\begin{proof}
Suppose, by contradiction, that $A$ and $B$ have 2 distinct points in common and call the set of these two points $E$. Let $A=E \cup \{u\}$ and $B=E \cup \{v\}$.
Since the rank of $q$ is 3, $ q \notin \langle \nu(E) \rangle$, but since by definition $q\in \langle \nu(A) \rangle \cap \langle \nu(B) \rangle $ we have that
$\langle \nu(E)\rangle \subsetneq \langle \nu(A) \rangle \cap \langle \nu(B) \rangle $. Clearly $\langle \nu(E)\rangle$ is a line, therefore $\dim \langle \nu(A) \rangle \cap \langle \nu(B) \rangle > 1 $, but $\langle \nu(A) \rangle $ and $\langle \nu(B) \rangle  $ are both planes, so $\langle \nu(A) \rangle =\langle \nu(B) \rangle  $. In the plane $\langle \nu(A) \rangle $ we have two different lines: $\nu(E) $ and $ \langle \nu(u),\nu(v) \rangle$, which mutually intersect in at most a point $q'$. Remark that  $q' \notin X$ because otherwise the line $\langle \nu(E)\rangle $ would have at least $3$ points of rank 1 and so we would have $\langle \nu(E)\rangle \subset  X$, contradicting Remark \ref{retta}.  So $r_X(q')=2 $ and $\sharp \mathcal{S}(Y,q')\geq 2 $, by Proposition \ref{2fattorisigma2} we get that actually $q' \in \langle \nu(Y')\rangle $, where $Y'=\mathbb{P}^1 \times \mathbb{P}^1  $. But also $E, \{ u,v\} \subset Y' $, so $q \in \langle \nu(Y')\rangle  $, which contradicts the fact that $q $ has rank $3 $. 
 \end{proof}

An immediate corollary of Lemma \ref{cardinalitatra5e6}  is the following.

\begin{corollary}\label{solo5o6} If $q$ is a rank-3 tensor and $A$ and $B$ are two distinct sets evincing its rank, then the cardinality of $A \cup B$ can only be either 5 or 6.
\end{corollary}

This corollary turns out to be extremely useful for the proof of our main result, Theorem \ref{main_theorem}. We will be allowed to focus only on the structure of not-identifiable points of rank-3 with at least two decompositions $A$ and $B$ as in Corollary \ref{solo5o6}. This is the reason why we will study separately the case $\sharp A\cup B =5$ in Section \ref{section:5points} form the case $\sharp A\cup B=6$ in Section \ref{section:6points}.

\medskip

Another very useful behaviour that needs to be understood in order to study the identifiability of rank-3 tensors, is the structure of the not-independent sets of at most 3 rank-1 tensors. This is what is described by the following lemma.

\begin{lemma}\label{3points}
A set of points $E\subset Y\simeq \mathbb{P}^{n_1}\times \cdots \times \mathbb{P}^{n_k}$ of cardinality at most 3
does not impose independent conditions to multilinear forms over $Y_i\simeq \mathbb{P}^{n_1}\times \cdots \times \hat{\mathbb{P}^{n_i}} \times \cdots \times \mathbb{P}^{n_k}$, $i=1, \ldots , k$, (i.e.  $h^1(\mathcal{I}_E(\hat{\varepsilon}_i)) >0$)
 if and only if one of the following cases occurs:
\begin{enumerate}
\item\label{uno} $\sharp (E) =3$ and there is $j\in \{1,\dots ,k\}\setminus \{i\}$ such that $\sharp (\pi _h(E)) =1$ for all $h\notin \{i,j\}$; 
\item\label{due} there are $u, v \in E$ such that $u\ne v$ and $\eta _i(u)=\eta _i(v)$.
\end{enumerate}
\end{lemma}

\begin{proof} The fact that both items \ref{uno}. and \ref{due}. imply that  $h^1(\mathcal{I}_E(\hat{\varepsilon}_i)) >0$ is obvious. Let us describe the other implication.

By definition $H^0(\mathcal{O}_Y(\hat{\varepsilon}_i)) \cong H^0(\mathcal{O}_{Y_i}(1,\dots ,1)$, and  $\mathcal{O} _Y(\hat{\varepsilon}_i)$ is not a very ample line bundle. So we cannot be sure about the injectivity of the restriction $\eta_{i|E}$ of $\eta_i$ to the finite set $E$.\\
If $\eta_{i|E}$ is not injective one immediately gets that $h^1(\mathcal{I}_E(\hat{\varepsilon}_i) )>0$. Moreover if $\eta_{i|E}$ is not injective it means that there are 2 distinct points of $E$, say $u$ and $v$ which are mapped by $\eta_i$ onto the same point, i.e. we are in item \ref{due}. of this lemma.

Now assume that $\eta _{i|E}$ is injective (i.e. we are not in item \ref{due}.). This implies that $\sharp E = \sharp \eta_i(E)$. We have by hypothesis that $h^1(\mathcal{I}_E(\hat{\varepsilon}_i))>0$. Since 
by definition $h^1(\mathcal{I}_E(\hat{\varepsilon}_i))
= h^1(Y_i, \mathcal{I}_{\eta_i(E)}(1,\dots ,1))$ we have that $\eta_i(E)$ does not impose independent conditions to the multilinear forms over $Y_i$, therefore $\sharp (\eta_i(E))\geq 3$ which clearly implies that $\sharp (\eta_i(E))=3$ since by hypothesis the cardinality of $E$ is at most 3.
Now $\eta_i(E)$ is a set of 3 distinct points on $Y_i$  which does not impose independent conditions to the multilinear forms over $Y_i$, and
$\mathcal{O}_{Y_i}(1,\dots ,1)$ is very ample, 
therefore the $3$ points of $\eta_i(E)$  must be mapped to collinear points by the Segre embedding $\nu _i$ of $Y_i$. Hence, by the structure of the Segre variety $\nu_i(Y_i)$, we get that $\langle \nu _i(\eta_i(E))\rangle \subseteq \nu _i(Y_i)$ and  there is $j\in \{1,\dots ,k\}\setminus \{i\}$ such that $\sharp (\pi _h(\eta_i(E))) =1$ for all $h\notin \{i,j\}$.
Since $h\ne i$, we have $\pi _h(\eta_i(E)) =\pi _h(E)$.
\end{proof}

\section{Two different solutions with one common point}\label{section:5points}

We have seen in Corollary \ref{solo5o6} that if a rank-3 tensor $q$ is not-identifiable and $A$, $B$ are two sets of points on the Segre variety computing its rank, then $\sharp A\cup B$ can only be either 5 or 6. This section is fully devoted to the case in which $\sharp A\cup B=5$, i.e. $A$ and $B$ share only one point  and call it $p$:
\begin{equation}\label{A'B'}S:= A\cup B, \; \;
    \sharp S=5, \; \;  A\cap B =\{p\} \hbox{ and }  A'=A \setminus \{ p \}, \; \;B'=B\setminus \{ p \}.
\end{equation}

The matrix case is well known, therefore we will always assume that $q$ is an order-$k\geq 3$ tensor, i.e. $q\in \langle \nu(Y) \rangle$ with $Y=\prod_{i=1}^k\mathbb{P}^{n_i}$ and $k\geq 3$.

We will study separately the cases in which:
\begin{itemize}
    \item $Y$ contains at least one factor of projective dimension 2 and all the others of dimension either 1 or 2 (Proposition \ref{prop:P2P1i});
\item   $Y$
is a product of $\mathbb{P}^1$'s only (see Proposition \ref{prop_P1}).

\end{itemize}
This will completely cover the cases of not-identifiable rank-3 tensors with  the condition \eqref{A'B'} since, by Remark \ref{remarkconcision}, the concise Segre of a rank-3 point $q$  is $X_q= \nu(\mathbb{P}^{n_1} \times \cdots \times  \mathbb{P}^{n_k} )$, with $n_1, \dots ,n_k \in \{1,2 \}$.

\begin{proposition}\label{prop:P2P1i}
Let $Y$
be the multiprojective space with at least 3 factors and at least one them of projective dimension 2, i.e. $Y=\mathbb{P}^2\times \mathbb{P}^{n_2}\times \cdots \times \mathbb{P}^{n_k}$ with $n_i\in \{1,2\}$ for $i=1, \ldots , k$ and $k\geq 3$. Let $q\in \sigma_3^0(\nu(Y))$, with $\nu(Y)$ the concise Segre of $q$. If there exist two sets $A,B\in \mathcal{S}( Y,q)$ evincing the rank of $q$ such that $\sharp A\cap B=1$ then $q$ is as in  Proposition \ref{x1.1}.
\end{proposition}

\begin{proof}
Consider a divisor $M\in \vert \mathcal{O}_Y(\varepsilon_1) \vert $ containing $A'=A\setminus \{ p\} $. By Concision/Autharky $S \nsubseteq M $, so, by Lemma \ref{lemmariformulato}, either $h^1(\mathcal{I}_{S \setminus S \cap M}(\hat{\varepsilon}_1))>0 $ or $p \notin M $ and $A'\cup B'\subset M $. We study separately the two cases.

\begin{enumerate}[leftmargin=
+.2in]
\item \label{b.1} 
First assume $h^1(\mathcal{I}_{S \setminus S \cap M}(\hat{\varepsilon}_1))>0 $.

The divisor $M$ contains $A'$ by definition so 
 $\sharp (S \setminus S\cap M) \leq 3$, moreover, if we define $Y_1:=\mathbb{P}^{n_2}\times \cdots \times \mathbb{P}^{n_k}$ with $n_i=1,2$ for $i=2, \ldots, k$, we have that $\mathcal{O}_{Y_1}(1,\ldots , 1) $ is very ample, 
therefore we can apply Lemma \ref{3points} and say that one of the following occurs:

\begin{enumerate}[label=\textnormal{(\roman*)}]
\item \label{bcaso1} $\sharp (S\setminus S\cap M)=3 $ and there exists a projection $\pi_i $, with $ i\in \{2, \ldots ,k \}$ such that $\sharp (\pi_i(S\setminus S\cap M))=1 $;
\item \label{bcaso2} There exist $u,v \in (S \setminus S\cap M) $ such that $ u\neq v$ and $\eta_1(u)=\eta_1(v). $
\end{enumerate}
We remark that case \ref{bcaso2} implies that $\pi_i(u)=\pi_i(v) $ for all $i >1 $.
Since $M  $ contains $A' $, we have that $S \setminus S \cap M =\{u,v\} \subseteq B $, we can exclude case \ref{bcaso2} thanks to Remark \ref{retta}.

So only case \ref{bcaso1} is possible. Since $\sharp (S\setminus S\cap M)=3 $ we have that  $S \setminus S \cap M=B $ 
and there exists an index $i \in \{ 2,\ldots , k \} $ such that $\sharp \pi_i(S\setminus S\cap M)=1 $. The fact that there is an $i\in \{2,\ldots , k\}$ such that $\sharp (\pi _i(B)) =1$,
means that $B$ only depends by $ k-1$ factors, contradicting Autarky.

\item \label{b.2} Now assume $A'\cup B'\subset M$.

Let $Y''$ be the minimal multiprojective space contained in $M$ and containing $A'\cup B'$. Since $q\in \langle \langle \nu (Y'')\rangle\cup \{p\})\rangle$ and $p\notin Y''$, there is a unique $o\in \langle
\nu (Y'')\rangle$ such that
$q\in \langle \{\nu ({p}),o\}\rangle$. Since $\langle \nu (A)\rangle$ (resp. $\langle \nu (B)\rangle$) is a plane containing
$\nu ({p})$ and $q$, there is a unique $o_1\in \langle \nu (A')\rangle$ (resp. $o_2\in \langle \nu (B')\rangle$) such that
$q\in \langle \{\nu ({p}),o_1\}\rangle$ (resp. $q\in \langle \{\nu ({p}),o_2\}\rangle$). The uniqueness of $o$ gives
$o=o_1=o_2$. Since $o_1=o_2$, we get a tensor of rank $2$ with $A'$ and $B'$ as solutions. Thus $ q$ is as described in Proposition \ref{x1.1}. 
\end{enumerate}
\vspace{-0.7cm}
\end{proof}

\begin{proposition}\label{prop_P1}
Let $Y=(\mathbb{P}^1)^{\times k}$ with $k\geq 3$ and let $q\in \sigma_3^0(\nu (Y))$ be such that there exist two different sets $A,B\in \mathcal{S}(Y,q)$ with the property $\sharp(A\cup B)=5$,  where $\nu(Y)$ is the concise Segre of $q$. Then $k$ can only be either 3 or 4.  If $k=3$ then $q$ belongs to a tangent space of $\nu((\mathbb{P}^1)^{\times 3})$ and $\dim (\mathcal{S}(Y,q))\geq 2$. If $k=4$ then $\dim (\mathcal{S}(Y,q))\geq 1$.
\end{proposition}

\begin{proof}
If $k=3$ then the only rank-3 tensors in $\langle \nu(\mathbb{P}^1)^{\times 3})\rangle$ are those belonging to the tangential variety of the Segre variety (cf. \cite{Buc,bb3}) for which $\dim (\mathcal{S}(Y,q))\geq 2$ (cf. \cite{AOP,BalBer3,CGGrank, CGGtecnica}).

The case $k=4$ is covered by Remark \ref{allP14}.

Assume $k>4$ and write $Y=\prod_{i=1}^k\mathbb{P}^1_i $. Let $S=A\cup B$ as in 
\eqref{A'B'}.
\\
We build a recursive set of divisors in order to being able to cover the whole set $S$ as follows.
Let $o_i\in \mathbb{P}^1_i$, $i=2,3,4$ be such that:
\begin{enumerate}[leftmargin=
+.7in]
    \item[$1^{\mathrm{st}}$ divisor:] $\pi _4^{-1}(o_4)\cap S \ne \emptyset$ and call $M_4:=\pi_4^{-1}(o_4)$;
    \item[$2^{\mathrm{nd}}$ divisor:] $\pi _3^{-1}(o_3)\cap (S \setminus (S\cap M_4)) \neq \emptyset$ and call $M_3:=\pi_3^{-1}(o_3)$.
        \item[$3^{\mathrm{th}}$ divisor:] If $M_3\cup M_4$ already covers the whole $S$ (i.e. $S\subset M_3\cup M_4 $), set $M_2$ to be any divisor $M_2\in |\mathcal{O} _Y(\hat{\varepsilon}_2)|$.
        \item[$3^{\mathrm{th}}$ divisor:]     Otherwise, if $S\nsubseteq M_3\cup M_4$, choose $o_2\in \mathbb{P}^1_2$  such that $\pi_2^{-1}(o_2)\cap(S \setminus S\cap (M_3\cup M_4))\neq \emptyset$ and set 
    $M_2:=\pi_2^{-1}(o_2)$.

\end{enumerate}

Now it may happen that either $S\subset M_2\cup M_3 \cup M_4$ or not. We study those two cases in \ref{contained} and \ref{notcontained} respectively.

\begin{enumerate}[leftmargin=+.2in, label=(\alph*)]
    \item \label{contained}

Here we assume that  $S\subset M_2\cup M_3 \cup M_4$. Since $\sharp (S)=5$ there is at least one of the $M_i$'s containing at
least two points of $S$, and there are two of the $M_i$'s whose union contains at least 4 points of $S$: wlog we may assume that $\sharp (S \cap (M_3\cup M_4))\geq 4$.
\begin{itemize}[leftmargin=+.1in]
\item
Assume $\sharp (S\cap (M_3\cup
M_4))= 4$. Since $\mathcal{O} _Y(1,1,0,0,\dots )$ is globally generated, we have that $h^1\big(\mathcal{I} _{S\setminus S\cap (M_3\cup
M_4)}(1,1,0,0,1,1,\dots )\big) =0$,
contradicting Lemma \ref{lemmariformulato}.

\item 
Assume $S\subset M_3\cup M_4$. Therefore there is one of the $M_i$'s containing at least 3 points of $S,$ let $\sharp (M_4\cap S) \ge 3$.
Since $S\nsubseteq M_4$, we get $h^1\big(\mathcal{I} _{S\setminus S\cap M_4}(\hat{\varepsilon}_4)\big) >0$ (by Lemma \ref{lemmariformulato}), hence $\sharp (S\setminus S\cap M_4) =2$ and 
\begin{equation}\label{uv}
S\setminus S\cap M_4=\{u,v\}
\hbox{ with }
\pi _i(u)=\pi _i(v), \; \forall
i\ne 4.\end{equation}
Since $h^1\big(\mathcal{I} _{S\setminus S\cap M_3}(\hat{\varepsilon}_3)\big) >0$ (again by Lemma \ref{lemmariformulato}, we get that either there are $w, z\in S\setminus S\cap
M_3$
such that $w\ne z$, $\pi _i(w)=\pi _i(z)$ for all $i\ne 3$ or $\nu _4(\eta _4(S\cap M_4))$ (remind Notation \ref{pi}) is made by $3$ collinear points, say
with a line corresponding to the $i$-th factor. The latter case cannot arise because $S$  does not depend only on the third,
fourth and $i$-th factor of $Y$. Thus there exist
\begin{equation}\label{wz}
w, z\in S\setminus S\cap
M_3 \hbox{ such that } w\ne z, \, \pi _i(w)=\pi _i(z) \; \forall i\ne 3.
\end{equation} 
In \eqref{uv} and \eqref{wz} we have $4$ distinct points $u, v, w, z$ such that $\sharp (\pi _5(\{u,v,w,z\})) =1$. Take $M_5\in |\mathcal{O}
_Y(\varepsilon _5)|$ containing
$\{u,v,w,z\}$. Since $h^1\big(\mathcal{I} _{S\setminus S\cap M_5}(\hat{\varepsilon}_5)\big) =0$, Autarky and Lemma \ref{lemmariformulato} give a
contradiction.
\end{itemize}

\item \label{notcontained}
Assume $S\nsubseteq M_2\cup M_3\cup M_4$. By Lemma \ref{lemmariformulato} we get $h^1(\mathcal{I} _{S\setminus S\cap (M_2\cup M_3\cup M_4)}(1,0,0,0,1,1,\dots ))>0$. Thus $\sharp (S\setminus (M_2\cup M_3\cup M_4)) =2$, say $S\setminus (M_2\cup M_3\cup M_4) =\{u,v\}$
and $\pi _i(u) =\pi _i(v)$ for all $i\neq2,3,4$. But in this case it is sufficient to change the orginal choice of $o_4$ and take as $o_4$ the point $\pi _4(u)$ and the the new divisor $M_4$ will contain 2 points of  $S$, i.e. $u,v$ therefore we are able to
 get new divisors
$M_2$, $M_3$ with the same contruction as above leading to the case $S\subset M_2\cup M_3\cup M_4$ excluded in step \ref{contained}. 
\end{enumerate}
\vspace{-0.7cm}
\end{proof}

\section{Two disjoint  solutions}\label{section:6points}

We have seen in Corollary \ref{solo5o6} that if a rank-3 tensor $q$ is not-identifiable and $A$, $B$ are two sets of points on the Segre variety computing its rank, then $\sharp A\cup B$ can only be either 5 or 6. This section is fully devoted to the case in which $\sharp A\cup B=6$, i.e. $A$ and $B$ are disjoint:
\begin{equation}\label{disjoint}S:= A\cup B, \; \;
    \sharp S=6, \; \; A:=\{a_1,a_2,a_3\}, B:=\{b_1,b_2,b_3\} \; \;  A\cap B =\emptyset.
\end{equation}

First of all let us show that if $q$ is a rank-3 tensor whose concise Segre $\nu(Y)$ has at least two factors of projective dimension $2$, it  never happens that in $\mathcal{S}(Y,q)$ there are two disjoint sets.

    \begin{remark}\label{remark2pti}
    Let $Y=(\mathbb{P}^{2})^{\times k_1}\times (\mathbb{P}^{1})^{\times k_2}$ and $S\subset Y$ a set of 6 distinct points. Consider $I\subseteq \{k_1+1, \ldots , k_1+k_2\}$  and $\varepsilon:=\sum_{i\in I}\varepsilon_i$. 
    Suppose there exists a divisor $ M\in |\mathcal{O}_Y(\varepsilon)|$ intersecting $S$ in 4 points. Call $\{u,v\}:=S\setminus (S\cap M)$. In this setting one can apply Lemma \ref{lemmariformulato} and get that $h^1\big(\mathcal{I}_{\{u,v\}}(\widehat{\varepsilon}
    )\big)>0 $ (where $\widehat{\varepsilon}$ is a $(k_1+k_2)$-uple with $0$'s in position of the indices appearing in $\varepsilon$ of $I$ and 1's everywhere else) and $\pi_h(u)=\pi_h(v)$ for any   $h\in \{1,\ldots , k_1+k_2\}\setminus I$.
\end{remark}

\begin{proposition}\label{6pP2P2Pn}
Let $Y $ be a multiprojective space with at least three factors and at least two of them of projective dimension $2$, i.e. $ Y=\mathbb{P}^2\times \mathbb{P}^2\times \mathbb{P}^{n_3}\times \cdots \times \mathbb{P}^{n_k}$ with $n_i\in \{ 1,2\} $ for  $i=1,\dots,k$ and $k\geq 3 $. Let $q\in \sigma_3^0(\nu(Y))$, with $\nu(Y)$ the concise Segre of $q$. If $A,B\in \mathcal{S}( Y,q)$ evince the rank of $q$, then $A$ and $B$ cannot be disjoint.
\end{proposition}

\begin{proof}
The proof is by absurd: assume that there exist $A,B\in \mathcal{S}( Y,q)$ with $A\cap B= \emptyset$. By Remark \ref{sulP2aebindip} we have that
 $\langle \pi _i(A)\rangle =
\langle \pi _i(B)\rangle =\mathbb{P}^2$ for $i=1,2$. 
Fix $W\in |\mathcal{I}_B(\varepsilon _2+\varepsilon _3)|$ (it exists, because $h^0(\mathcal{O}_Y(\varepsilon _2+\varepsilon _3)) = h^0(\mathbb{P}^2\times
\mathbb{P}^{n_3}, \mathcal{O}_{\mathbb{P}^2\times
\mathbb{P}^{n_3}}(1,1)) = 3(n_3+1)>4$). Since $\pi _{1|A}$ is injective, $h^1(\mathcal{I} _A(\varepsilon _1))=0$. Thus $S\subset W$ by Lemma \ref{lemmariformulato}.
In this way we have shown that 
\begin{equation}\label{eqtxt}
    \hbox{any divisor } D \in \vert \mathcal{O}_Y(\varepsilon_2+\varepsilon_3) \vert \hbox{ containing } B \hbox{ contains  also }A . \tag{*}
\end{equation}

\vspace{-1cm}

\begin{quote}

\begin{claim}\label{poker}$\pi_3(a_i)=\pi_3(b_i)$ where $a_i,b_i$ are as in \eqref{disjoint}, for $i=1,2,3$.
\end{claim}

\end{quote}

The proof of this claim can be repeated verbatim for all the other projections with only one caution that we will highlight in the sequel. 
Therefore, by repeating the argument  for all the projections, we will get that $\pi_j(a_i)=\pi_j(b_i)$ for $i=1,2,3$ and for $j=1, \ldots, k$ which is a contradiction with $A$ and $B$ being distinct. This will conclude the proof.

\begin{quote}
\begin{proof}[Proof of the Claim \ref{poker}]
Take a general hyperplane $J_3\subset\mathbb{P}^{n_3}$ containing $\pi _3({b_i})$, (where the $b_i$'s are as in  \eqref{disjoint}, $i=1,2,3$)
by genericity we may assume that if $n_3=2$ then $J_3$ is a line which does not contain any other point of that projection. Set $M_3:= \pi _3^{-1}(J_3)$. Take a line
\begin{equation}\label{L2M2}
L_2\subset \mathbb{P}^2 \hbox{ containing }\{\pi_2(b_j),\pi _2(b_k)\} \hbox{ with } j,k\neq i \hbox{ and set } M_2:= \pi
_2^{-1}(L_2).  \tag{**}
\end{equation}
We have
$B\subset M_2\cup M_3\in |\mathcal{O} _Y(\varepsilon _2+\varepsilon _3)|$.  Thus from \eqref{eqtxt} we get that $M_2\cup M_3$ contains also $ A$. Since $A\nsubseteq M_2$ by
Autarky, there is $a\in A\cap M_3$, i.e. there is $a\in A$ such that 
\begin{equation}\label{pi3a}
    \pi _3(a) =\pi _3({b_i})
\end{equation}
(in fact if $n_3=1$ it is trivial, if $n_3=2$ then we have already remarked that $\pi_3(b_i)$ is the only point of $J_3$ belonging to $\pi_3(Y)$).
Since $\pi_{i\vert  A} $ is injective for $i=1,2 $ (cf. Remark \ref{sulP2aebindip}), 
the points of $A$ projecting on $\pi_3(b_i)$ are different for different $i$'s except if there are $b_i\neq b_j$ such that $\pi_3(b_i)=\pi_3(b_j)$. Suppose that this is the case.
By Lemma \ref{3points} we get that  $\sharp(S\setminus S\cap M_3)=2$.
 Thus if for $i\neq j$
 $\pi_3(b_i)=\pi_3(b_j) $
there are 2 points of $A$ and 2 points of $B$ in $M_3$, i.e. $\sharp(S\cap M_3)= 4$. Suppose that $ S\cap M_3=\{ a_3,b_3,a_2,b_2\}$.
By \cite[Lemmas 2.4  and 2.5]{BalBerChristGes} (also \cite[Lemma 5.1, item (b)]{BalBer2})  $h^1(\mathcal{I}_{S\setminus S\cap M_3}(\hat{\varepsilon}_3))>0 $, i.e. $ \pi_i(a_1)=\pi_i(b_1)$ for all $i\neq 3 $. 
This is a contradiction since we already know that  $\pi_3(a_2)=\pi_3(b_2)$ and we would have $a_2=b_2 $, which contradicts the assumption that  $A\cap B=\emptyset.$ 

Therefore the points $a\in A$ of \eqref{pi3a} are all different for different choices of $i$'s. So we may assume that $\pi_3(a_i)=\pi_3(b_i)$ for $i=1,2,3$ and the $\pi_3(b_i)\neq \pi_3(b_j)$ for $i\neq j$.
\end{proof}
\end{quote}

The argument of the proof of Claim \ref{poker} can be repeated verbatim for all the others $\pi_j$'s with the only caution that when we do the case $j=2$ we have to use a line $L_1\subset \mathbb{P}^2$ containing $\{\pi_1(b_j),\pi _1(b_k)\}$ with $j,k\neq i$ and set $M_1:= \pi
_1^{-1}(L_1)$ instead of $M_2$ and $L_2$ in \eqref{L2M2}. Moreover \eqref{eqtxt} clearly holds if we replace the $\varepsilon_2$ with $\varepsilon_1$ and $\varepsilon_3$ with $\varepsilon_j$ for any $j=3, \ldots , k$. As already highlighted this concludes the proves since $\pi_j(a_i)=\pi_j(b_i)$ for $i=1,2,3$ and for $j=1, \ldots, k$ which is a contradiction with $A$ and $B$ being distinct.  
\end{proof}

This shows that under the assumption \eqref{disjoint}, we can exclude the case where the Segre variety has at least two factors of projective dimension 2. 

\smallskip

Let us focus on the 4-factors case.

\def\lemmareimmesso{

\begin{lemma}\label{alphamax} Let $Y=\mathbb{P}^2\times \mathbb{P}^1\times \mathbb{P}^1\times \mathbb{P}^1 $, let $q\in \sigma_3^0(\nu(Y))$, with $\nu(Y)$ the concise Segre of $q$ and such that  there exist two disjoint sets $A,B\in \mathcal{S}(Y,q)$. If one among the projections $\eta_{2\vert S} $, $\eta_{3\vert S}$ and  $\eta_{4\vert S}$ is injective, then there exist a point $o$ in one of the $\mathbb{P}^1$'s such that it's preimage via either $\pi_2$ or $\pi_3$ or $\pi_4$ intersects $S:=A\cup B$ in 3 points and this is the maximum number of points that $S$ can share with such a preimage.
\end{lemma}

\begin{proof} \begin{AB}
Firstly define recursively for $j=4,3$ the integers such that the preimages of points $o\in \mathbb{P}^1$ intersect maximally the set $S$:

\begin{equation}\label{alpha4}
\alpha_4:=\max \{\sharp(\pi_i^{-1}(o)\cap S)\}_{o \in \mathbb{P}^1;i=2,\ldots , 4}.
\end{equation}
By rearranging if necessary, we can assume that the index $i=2, \ldots , 4$ realizing $\alpha_4$, is $i=4$. Then define 
\begin{equation}\label{alpha3}
\alpha_3:=\max \{\sharp(\pi_i^{-1}(o)\cap (S\setminus (S\cap K_4)))\}_{o \in \mathbb{P}^1;i=2,3}.
\end{equation}
By rearranging if necessary, we can assume that the index $i=2,3$ realizing $\alpha_3$, is $i=3$. Finally define
\begin{equation}\label{alpha2}
\alpha_2:=\max \{\sharp(\pi_2^{-1}(o)\cap (S\setminus (S\cap K_4\cup K_3)))\}_{o\in \mathbb{P}^1}.
\end{equation}

Now let $o_j\in \mathbb{P}^1$, $j=2,3,4$ be the points realizing $\alpha_2, \alpha_3, \alpha_4$ respectively, and call 
\begin{equation}\label{Kj}
   K_j:=\pi_j^{-1}(o_j) \hbox{ for } j=2,3,4.
\end{equation}  

\begin{quote}
    Remark that by Autarky assumption ${1\leq \alpha_3 \leq \alpha_4\leq 5 }$.
\end{quote}
It is easy to see that $\alpha_4$ cannot be 5. In fact if $\alpha_4=5$, then  $\sharp (S\setminus S\cap K_4)=1 $ which implies that $h^1(\mathcal{I}_{S\setminus S\cap K_4}(1,1,1,0))=0 $, which is a contradiction with \cite[Lemmas 2.4  and 2.5]{BalBerChristGes} (also \cite[Lemma 5.1, item (b)]{BalBer2}).
\begin{quote}
    So the possibilities for the $\alpha_j$'s are  $1\leq \alpha_3 \leq \alpha_4 \leq 4$.
\end{quote}
Now we exclude the cases $\alpha_i=1$, $i=3,4$ by showing that $\alpha_3\neq 1$.

\def\alphaneq1{
    If $\alpha_4=1$  the $\pi_i|_S$'s, for $i=2,3,4$, are all injective. By  Remark \ref{sulP2aebindip}, both $\pi_1|_A$ and $\pi_1|_B$ are also injective, so each fiber of an element of $\pi_{1\vert S} $ has at most cardinality $2$. Thus we may find $ M_i\in \vert \mathcal{O}_Y(\varepsilon_i) \vert$, with $i=1,2,3$, such that $\sharp (S \cap (M_2\cup M_3)     ) =2$ and such that $\sharp (S \cap (M_1\cup M_2\cup M_3) ) =4 $. Set
    \begin{equation}\label{F}
        F:= M_1\cup M_2\cup M_3 
    \end{equation} and let $\{ u,v\} :=S\setminus S\cap F$. By \cite[Lemmas 2.4  and 2.5]{BalBerChristGes} (also \cite[Lemma 5.1, item (b)]{BalBer2}) we get in particular that \begin{AB} $h^1(\mathcal{I}_{\{u,v\}}(\varepsilon_4))>0 $,\end{AB} i.e. $\pi_4(u)=\pi_4(v) $ which contradicts injectivity of $\pi_{4\vert S} $ that we have for free since we are assuming that $\alpha_4=1$. So $\alpha_4 \neq 1$.
    %
    %
    %
}
First of all, remark that if $\alpha_3=1$, then $\pi_3|_S$ is injective. So the idea is to
build a divisor $F$ such that $\sharp (S\setminus F)=2$ and applying Remark \ref{remark2pti} to $F$: the existence of such an $F$ will contradict 
the injectivity of  $\pi_{3\vert S} $. 
The divisor $F$ 
is as follows: 
$F=K_4 \cup H_2 \cup H_3$
where $K_4$ is as in \eqref{Kj}  and, for $i=2,3$,
$H_i \in |\mathcal{O}_Y(\epsilon_i)|
\hbox{ such that }H_i\cap (S\setminus S \cup K_4) \neq \emptyset$. 
This shows $\alpha_3 \neq 1$.
\begin{quote}
    We are therefore left with $1 < \alpha_3 \leq \alpha_4 =2,3,4$.
\end{quote}

Suppose that $\alpha_3=\alpha_4=2$.
By the construction of the $K_i$'s in \eqref{Kj} for $i=2,3,4$, it's easy to show that $$S=\mathbin{\coprod}_{i=2}^4 S\cap K_i.$$
The only non-obvious fact is that $\sharp (S\cap K_2)=2$. On one hand we may always take $H\in \vert \mathcal{O}_Y(\varepsilon_2)\vert$ such that $   \sharp(S\setminus S\cap (K_3 \cup K_4)) \cap H)\neq 0$, so such a $H$  intersects $S$ non trivially and $K_2$ is among those $H$'s. On the other hand we may also suppose
\begin{equation}\label{dag}
    \alpha_2\neq 1\tag{\dag}
\end{equation}
(where $\alpha_2$ is defined in \eqref{alpha2})
because if $\alpha_2=1$ then
 the divisor $K_4\cup K_3\cup K_2\in |\mathcal{O}_Y(\hat \varepsilon_1)|$ therefore 
$h^1(\mathcal{I}_{S\setminus (K_4 \cup K_3 \cup K_2)}(\varepsilon_1))=0$ which is a 
contradiction with \cite[Lemmas 2.4  and 2.5]{BalBerChristGes}
(also \cite[Lemma 5.1, item (b)]{BalBer2}). 

So, since $S=\mathbin{\coprod}_{i=2}^4 S\cap K_i$ and $\sharp (S\cap K_i)=2$ for $i=2,3,4$, we can apply Remark  \ref{remark2pti} 
separately to the divisors $K_i\cup K_j $ with $i\neq j$ and
get that $h^1(\mathcal{I}_{S\cap K_i}(\varepsilon_1+\varepsilon_i))>0 $ for $i=2,3,4$ and so 
$\pi_1(S\cap K_i)=1$ for $i=2,3,4$.
In order to get a contradiction it is sufficient to apply again Remark \ref{remark2pti} to $\pi_1^{-1}(\langle \pi_1(S\cap K_3) , \pi_1(S\cap K_2)  \rangle) $. This shows that 
$\sharp (\pi_i(S\cap K_4))=1 $ for  $i=2,3,4$. Now since also $\sharp (\pi_1(S\cap K_4))=1 $, then $\sharp (S\cap K_4)=1 $, which is a contradiction with the assumption $\alpha_3=2 $. 

\begin{quote}
   This proves that  $\alpha_3>1$, and $2<\alpha_4=3,4$.
\end{quote}

The same argument used to prove \eqref{dag} allows to exclude also the cases $(\alpha_2, \alpha_3, \alpha_4)=(1,2,3)$.
Moreover also $(\alpha_3,\alpha_4)=(2,4)$ can be excluded using the same argument of the case $(\alpha_2,\alpha_3,\alpha_4)=(2,2,2)$ above applying Remark \ref{remark2pti} since if $(\alpha_3,\alpha_4)=(2,4)$ we have that $K_4$ plays the role of $M$ in the remark.

\begin{quote}
    We are therefore left with the unique possibility of $(\alpha_3, \alpha_4)=(3,3)$.
\end{quote}

\end{AB}

\end{proof}

}

\begin{proposition}
\label{6pP2P1P1P1P1}
Let $Y=\mathbb{P}^2\times\mathbb{P}^1\times\mathbb{P}^1 \times \mathbb{P}^{1}$. Let $q\in \sigma_3^0(\nu(Y))$, with $\nu(Y)$ the concise Segre of $q$. There do not exist two disjoint sets $ A,B \in \mathcal{S}(Y,q)$ evincing the rank of $ q$.
\end{proposition}
 
\begin{proof}

Assume by contradiction that there exist two disjoint sets $A,B \in \mathcal{S}(Y,q) $ evincing the rank of $q$ and moreover assume that no $\eta_{i\vert S}$ is injective, for $i=2,3,4$.

By Remark \ref{retta}, for each $ i=2,3,4$ there exists $a \in A $, $b \in B $ such that $\eta_i(a)=\eta_i(b) $.
Fix $H:=\pi_1^{-1}(L) $, where $L\subset \mathbb{P}^2 $ is a line containing $ \pi_1(a_1)$ and $\pi_1(a_2) $, where $a_1,a_2\in A $. Since we assumed that no $ \eta_{i\vert S}$ is injective, then there exist $b_1,b_2 \in B$ such that $\pi_1(a_i)=\pi_1(b_i) $, for $i=1,2 $. 
Thus $H\supset \{a_1,a_2,b_1,b_2 \} $ and by Autarky $S \not\subset H $, so there is at least an element of $S $ out of $ H$, e.g. $a_3\in S\setminus \{a_1,a_2,b_1,b_2 \} $. 
Thus we have $h^1\big(\mathcal{I}_{S\setminus S\cap H}(0,1,1,1)\big)=0 $ contradicting Lemma \ref{lemmariformulato}.
So there exists at least one integer $h\in \{2,\dots,4\} $ such that $ \eta_{h\vert S}$ is injective.

Firstly define recursively the integers such that the preimages of points $o\in \mathbb{P}^1$ intersect maximally the set $S$:

\begin{equation}\label{alpha4}
\alpha_4:=\max \{\sharp(\pi_i^{-1}(o)\cap S)\}_{o \in \mathbb{P}^1;i=2,\ldots , 4}.
\end{equation}
By rearranging if necessary, we can assume that the index $i=2, \ldots , 4$ realizing $\alpha_4$, is $i=4$. Call $ K_4:=\pi_4^{-1}(o)$. Then define 
\begin{equation}\label{alpha3}
\alpha_3:=\max \{\sharp(\pi_i^{-1}(o)\cap (S\setminus (S\cap K_4)))\}_{o \in \mathbb{P}^1;i=2,3}.
\end{equation}
By rearranging if necessary, we can assume that the index $i=2,3$ realizing $\alpha_3$, is $i=3$. Call $ K_3:=\pi_3^{-1}(o)$. Finally define
\begin{equation}\label{alpha2}
\alpha_2:=\max \{\sharp(\pi_2^{-1}(o)\cap (S\setminus (S\cap K_4\cup K_3)))\}_{o\in \mathbb{P}^1}.
\end{equation} 
So if we denote by $o_j\in \mathbb{P}^1$, $j=2,3,4$ the points realizing $\alpha_2, \alpha_3, \alpha_4$ respectively then we call 
\begin{equation}\label{Kj}
   K_j:=\pi_j^{-1}(o_j) \hbox{ for } j=2,3.
\end{equation}  

\begin{quote}
    Remark that by Autarky assumption 
     ${1 \leq \alpha_3 \leq \alpha_4\leq 5 }$.
\end{quote}
It is easy to see that $\alpha_4$ cannot be 5. In fact if $\alpha_4=5$, then  $\sharp (S\setminus S\cap K_4)=1 $ which implies that $h^1(\mathcal{I}_{S\setminus S\cap K_4}(1,1,1,0))=0 $, which is a contradiction with Lemma \ref{lemmariformulato}.
\begin{quote}
    So the possibilities for  $\alpha_3$ and $\alpha_4$ are
    $1\leq  \alpha_3 \leq \alpha_4 \leq 4$.\end{quote}

Let us show that
\begin{equation}\label{dag}    \alpha_2\neq 1.\tag{\dag}
\end{equation}

Assume that $(\alpha_2, \alpha_3, \alpha_4)=(1,1,1) $. 
In such a case the divisor $ K_2\cup K_3\cup K_4\in \vert \mathcal{O}_Y(\hat{\varepsilon}_1) \vert $ would contain exactly 3 points of $ S$.
Moreover if $ h^1\big(\mathcal{I}_{S\setminus (S\cap K_2\cup K_3\cup K_4)}(\varepsilon_1)\big)>0$ then by Lemma \ref{3points} we would have a contradiction with $(\alpha_3,\alpha_4)=(1,1) $. Therefore if $(\alpha_2, \alpha_3, \alpha_4)=(1,1,1) $ we must have $ h^1\big(\mathcal{I}_{S\setminus (S\cap K_2\cup K_3\cup K_4)}(\varepsilon_1)\big)=0$, but this is a contradiction with Lemma \ref{lemmariformulato}. 
Thus if $ \alpha_2=1$ then $K_3\cup K_4 $ should contain at least 3 points of $ S$, i.e. $ \alpha_3\geq 1$ and $ \alpha_4\geq 2$.

\def\alphaneq1{
    If $\alpha_4=1$  the $\pi_i|_S$'s, for $i=2,3,4$, are all injective. By  Remark \ref{sulP2aebindip}, both $\pi_1|_A$ and $\pi_1|_B$ are also injective, so each fiber of an element of $\pi_{1\vert S} $ has at most cardinality $2$. Thus we may find $ M_i\in \vert \mathcal{O}_Y(\varepsilon_i) \vert$, with $i=1,2,3$, such that $\sharp (S \cap (M_2\cup M_3)     ) =2$ and such that $\sharp (S \cap (M_1\cup M_2\cup M_3) ) =4 $. Set
    \begin{equation}\label{F}
        F:= M_1\cup M_2\cup M_3 
    \end{equation} and let $\{ u,v\} :=S\setminus S\cap F$. By \cite[Lemmas 2.4  and 2.5]{BalBerChristGes} (also \cite[Lemma 5.1, item (b)]{BalBer2}) we get in particular that $h^1(\mathcal{I}_{\{u,v\}}(\varepsilon_4))>0 $, i.e. $\pi_4(u)=\pi_4(v) $ which contradicts injectivity of $\pi_{4\vert S} $ that we have for free since we are assuming that $\alpha_4=1$. So $\alpha_4 \neq 1$.
}
 Now assume that $(\alpha_2,\alpha_3)=(1,1)$. Then $\pi_{3|S}$ is injective. The idea is to build a divisor $F\in \vert \mathcal{O}_Y(\varepsilon) \vert$ with $\varepsilon=\sum_{i\in I} \varepsilon_i $, for some finite $I \in \{1,\dots,k \}$, such that $\sharp (S\setminus F\cap S)=2$ and apply Remark \ref{remark2pti} to $F$: the existence of such a $F$ will contradict the injectivity of  $\pi_{3\vert S} $. 
 Let 
$H_i \in |\mathcal{O}_Y(\varepsilon_i)|
\hbox{ such that }H_i\cap (S\setminus S \cap K_4) \neq \emptyset$ for $i=2,3$.
The divisor $F$ is either  
$F=K_4$, or $ F=K_4\cup H_3$ or $ K_4 \cup H_2 \cup H_3$  if $\alpha_4=4,3,2$ respectively.
The case $ (\alpha_2,\alpha_3,\alpha_4)=(1,2,2)$  can be easily excluded since $\sharp (S\setminus S\cap K_2\cup K_3\cup K_4)=1 $ and by Lemma \ref{lemmariformulato} we would have $h^1\big(\mathcal{I}_{S\setminus S\cap (K_2\cup K_3\cup K_4)}(\varepsilon_1)\big)>0 $, which is absurd.
For the same reason $ (\alpha_2,\alpha_3,\alpha_4)=(1,2,3)$ is also impossible because then $\sharp (S\cap (K_3\cup K_4))=5 $ and by Lemma \ref{lemmariformulato} we would have $h^1\big(\mathcal{I}_{S\setminus S\cap (K_3\cup K_4)}(1,1,0,0)\big)>0 $, which is a contradiction.
This shows $\alpha_2\neq 1 $.

\begin{quote}
    We are therefore left with $\alpha_2\neq 1 < \alpha_3 \leq \alpha_4 =2,3,4$.
\end{quote}

Suppose that $\alpha_3=\alpha_4=2$. With these assumptions one also gets $\alpha_2=2$. Indeed on one hand we just showed that we
may always take $H\in \vert \mathcal{O}_Y(\varepsilon_2)\vert$ such that $   \sharp(S\setminus S\cap (K_3 \cup K_4)) \cap H)\neq 0$, so such a $H$  intersects $S$ non trivially and $K_2$ is among those $H$'s. On the other hand $\alpha_2\neq 1$ by \eqref{dag}. So $\sharp (S\cap K_2)=2$.
By the construction of the $K_i$'s in \eqref{Kj} for $i=2,3,4$, it's easy to show that $$S=\mathbin{\coprod}_{i=2}^4 S\cap K_i.$$
So, since $S=\mathbin{\coprod}_{i=2}^4 S\cap K_i$ and $\sharp (S\cap K_i)=2$ for $i=2,3,4$, we can apply Remark \ref{remark2pti} separately to the divisors $K_i\cup K_j $ with $i\neq j$ and get that $h^1\big(\mathcal{I}_{S\cap K_i}(\varepsilon_1+\varepsilon_i)\big)>0 $ for $i=2,3,4$ and so $\pi_1(S\cap K_i)=1$ for $i=2,3,4$.
In order to get a contradiction it is sufficient to apply again Remark \ref{remark2pti} to $\pi_1^{-1}(\langle \pi_1(S\cap K_3) , \pi_1(S\cap K_2)  \rangle) $. This shows that $\sharp (\pi_i(S\cap K_4))=1 $ for  $i=2,3,4$. Now since also $\sharp (\pi_1(S\cap K_4))=1 $, then $\sharp (S\cap K_4)=1 $, which is a contradiction with the assumption $\alpha_3=2 $. 

\begin{quote}
   This proves that  $1<\alpha_2 \leq \alpha_3$, and $2<\alpha_4=3,4$.
\end{quote}

The case $(\alpha_3,\alpha_4)=(2,4)$ can be excluded using the same argument of the case $(\alpha_2,\alpha_3,\alpha_4)=(2,2,2)$ above applying Remark \ref{remark2pti} since if $(\alpha_3,\alpha_4)=(2,4)$ we have that $K_4$ plays the role of $M$ in the remark.
\begin{quote}
    We are therefore left with the unique possibility of $(\alpha_3, \alpha_4)=(3,3)$.
\end{quote}

\vspace{-0.8cm}

\begin{quote}

\begin{claim}\label{claim63}
$\sharp (\pi_2(S\cap K_4))=1 $.
\end{claim}

\begin{proof}[Proof of Claim \ref{claim63}:]
Since we are in the hypothesis $\alpha_4=3$, the projection of $S\cap K_4$  onto the first two factors of $Y$ is made by at most 3 points.

Suppose that such a projection is made by exactly 3 points. Call $Z $ the image of the projection of $S\cap K_4 $ onto the first two factors. Since $h^1(\mathbb{P}^2\times \mathbb{P}^1,\mathcal{I}_Z(1,1))>0 $ those points must lie on a line $L$ when applying the Segre embedding.
Moreover from Remark \ref{sulP2aebindip} we know that
$\pi_1( A) $ and $\pi_1( B) $ are sets of linearly independent points
 and since linear subspaces of the Segre variety are all contained in a factor, we get that $L \subset \mathbb{P}^2 $. Thus $ \sharp (\pi_2(S\cap K_4))=1$ proving the claim in this case.

If the projection of $S\cap K_4$ onto the first two factors is made by less than 3 points, there exist at least two points, $u,v \in S \cap K_4 $ such that they share the same image under the projection. 
Remark that if we consider $ E\subset S\cap K_4$ such that $\sharp E=2 $ and take $ T\in \vert \mathcal{I}_E(1,1,0,0) \vert$, then $T\supset S\cap K_4 $. Indeed if $ S\cap K_4\not\subset T$ then we have that $T\cup K_3 $ contains exactly five points of $S $, which leads to a contradiction because by Lemma \ref{lemmariformulato} we would have $h^1\big(\mathcal{I}_{S\setminus S\cap T \cup K_3}(\hat{\varepsilon}_3)\big)>0 $.
Therefore also the third point of $S\cap K_4 $ share the same image of $ u$ and $ v$ and we are done. 
\end{proof}
\end{quote}

Using the third factor instead of the second one, one gets $\sharp(\pi_3(K_4\cap S))=1 $ and since we assumed that $\alpha_4 $ is reached on the fourth factor we also have $\sharp(\pi_4(K_4\cap S))=1 $. The same argument can be applied to $S\cap K_3$ which leads to $\sharp(\pi_2(K_3\cap S))= \sharp(\pi_4(K_3\cap S))=1$.
Thus $ \sharp(\pi_i(K_4\cap S))= \sharp(\pi_i(K_3\cap S))=1$ for all $i>1 $ which contradicts Autarky.
\end{proof}

Since the identifiability of rank-3 tensors in $\langle \nu((\mathbb{P}^1)^{\times 4}) \rangle$ is already fully described by Remark \ref{allP14}, we are therefore done with the  order-4 tensors and we can focus on tensors of order bigger or equal than 5. So we will deal with $Y=\mathbb{P}^{n_1}\times (\mathbb{P}^1)^{l} $, with $n_1=1,2 $ and $l\geq4 $.

\begin{lemma}
\label{etainiettivi}

Let $q$ be a rank-3 tensor of order at least 5 and let $\nu(Y)$ be its concise Segre. If  there exist two disjoint sets $A,B\in \mathcal{S}(Y,q)$ as in \eqref{disjoint}, then there exists at least an index $i \in \{1,\dots, k \} $ such that $\eta_{i\vert S} $ and $\pi_{i\vert S}$ are injective. 

\end{lemma}

\begin{proof} 

\medskip

[Injectivity of $\eta_{i\vert S}$.]

Assume that no $\eta_{i\vert S} $ is injective, then by Remark \ref{retta} for any $i=1,\dots ,k $ there exist an element $a\in A$ and an element $b \in B$ such that $\pi_h(a) =\pi_h(b)$ for any $h\neq i $. It is easy to check that this condition, applied to two disjoint sets of 3 points each, and at least five $\eta_i$'s, imposes either that $A\cap B \neq \emptyset$ (contradiction) or that one of the two sets (either $A$ or $B$) depends only on 4 factors (contradicting Autarky). 

\def\tesi{
\begin{pp}
Since $\eta_{i\vert S} $ is not injective, for each $ i=1,\dots,3$ we have that $\pi_j(a_i) =\pi_j(b_i)$ for all $j\neq i $, with $ j=1,\dots,5$.\end{pp} The same holds for $\eta_{4\vert S} $: if $\eta_{4\vert S}(a_i)=\eta_{4\vert S}(b_i) $ for some $i \in \{1,2,3 \} $ we already get a contradiction with the condition $A\cap B=\emptyset $. We thus assume without loss of generality that $\eta_{4\vert S}(a_1)=\eta_{4\vert S}(b_2) $.\\
If $\eta_{5\vert S}(a_i)=\eta_{5\vert S}(b_i) $ with $i=1,2,3 $, as before we get a contradiction since $A\cap B=\emptyset $. \begin{AB}
If either $\eta_5(a_2)=\eta_5(b_1)$ or $\eta_5(a_1)=\eta_5(b_2) $
then $a_2=b_1 $ and $a_1=b_2$, which contradicts $A\cap B=\emptyset $. If either $\eta_5(a_3)=\eta_5(b_1) $ or $\eta_5(a_1)=\eta_5(b_3) $ 
we would get that $b_1,b_2,b_3 $ have the same component on the second factor and on the first one, respectively ?????, which contradicts Autarky. Finally if $a_2,b_3 $ or $a_3,b_2 $ have equal image through $\eta_5$, we would have that $a_1,a_2,a_3 $ share the same first component, contradicting Autarky.
\end{AB}}

\medskip

[Injectivity of $\pi_{i\vert S}$.]

Assume that $\eta_{i|S}$ is injective and that $\pi_{i\vert S}$ is not injective. If $i=1 $ and the first factor of $ Y$ is $ \mathbb{P}^2$ take $ H\in \vert\mathcal{O}_Y(\varepsilon_i) \vert $ as the preimage of a general line that contains exactly one point of $S $; otherwise take $H\in \vert \mathcal{O}_Y(\varepsilon_i) \vert$ as the preimage of a point of $S $. We remark that in both cases we get that $\sharp (\pi_i(S\cap H))=1$. Since by Autarky $S\not\subset H $, by Lemma \ref{lemmariformulato} we have that $$ h^1\big(\mathcal{I}_{S\setminus S\cap H}(\hat{\varepsilon}_i)\big)>0.$$
We distinguish different cases depending on $\sharp (S\setminus S\cap H) $.

\begin{enumerate}

    \item\label{sharp4}

    Assume $\sharp (S\setminus S\cap H)=4 $ and call $S':=\eta_i(S\setminus S\cap H) $; let $A' \subset S' $ such that $\sharp A'=2 $ and call $B':=S'\setminus A' $, so $\sharp B'=2 $.
    Since  $\eta_{i \vert S}$ is injective we have that $h^1(Y_i,\mathcal{I}_{S'}(\hat{\varepsilon_i}))=h^1\big(\mathcal{I}_{S\setminus S\cap H}(\hat{\varepsilon}_i)\big)>0 $. So $\langle \nu_i(A')\rangle \cap \langle \nu_i(B') \rangle \neq \emptyset $, which means that we have at least a point $q' \in \langle \nu_5(Y_i)\rangle$ of rank $2 $ for which $A' $ and $B' $ are different subsets evincing its rank. Thus by Proposition \ref{sigma20}, since $ \sharp \mathcal{S}(Y_i,q')>1$, the points in $A' $ and $B' $ only depend on two factors, i.e. $\sharp (\pi_j(S'))=1 $ for at least two indices $j\in \{1,\dots,k \}$. Without loss of generality assume it happens for $j=1,2 $. If the first factor of $Y$ is $\mathbb{P}^2 $, let $M_1\in \vert \mathcal{I}_{\pi_1(S')}(\varepsilon_1) \vert  $ be the preimage of a general line containing $\pi_1(S')$  and let $\{ M_2\}:=\vert \mathcal{I}_{\pi_2(S')}(\varepsilon_2) \vert $. Otherwise let $\{ M_j\}:=\vert \mathcal{I}_{\pi_j(S')}(\varepsilon_j) \vert $, for $j=1,2 $; in both cases then $h^1\big(\mathcal{I}_{S\setminus S\cap M_j}(\hat{\varepsilon}_j)\big)>0 $. So $S\setminus S\cap M_j=S\cap H $ and $\sharp (\eta_j(S\cap H))=1 $, for $j=1,2 $. If we call $S\cap H=\{u,v \} $, it follows that $\eta_1(u)=\eta_1(v) $ and $\eta_2(u)=\eta_2(v) $, so in particular we get that $ \pi_j(u)=\pi_j(v)$ for any $j$, which is a contradiction.
    
    \item\label{sharp3} Assume $\sharp (S\setminus S\cap H)=3 $. By Proposition \ref{3points}  
    there exists  $j\neq i$ such that  $\sharp (\pi_h(S\setminus S\cap H))=1$ for all $h\neq i,j $.
    For all $ h>1$ with $h\neq j,i $, since $h^0(\mathcal{O}_Y(\varepsilon_h))=2 $ we get $ h^0\big(\mathcal{I}_{S\setminus S\cap H}(\varepsilon_h)\big)=1$. Set $\{ M_h \}:=\vert \mathcal{I}_{S\setminus S\cap H}(\varepsilon_h) \vert  $, if $h=1 $ and the first factor is $\mathbb{P}^2 $ take $M_h\in \vert \mathcal{I}_{S\setminus S \cap H}(\varepsilon_h) \vert  $ as the preimage of a general line, otherwise call $\{ M_h \}:=\vert \mathcal{I}_{S\setminus S\cap H}(\varepsilon_h) \vert  $. Since we took $H $ such that $\sharp \pi_i(S\cap H)=1$, there exists at least an index $t\neq i $ such that $\sharp \pi_t(S\cap H)\geq 2 $. Thus we can find $D\in \vert \mathcal{O}_Y(\varepsilon_t)\vert $ containing exactly one point of $S\cap H $.\\ For all $s\neq t$ set $W_s:=M_s\cup D $, so $ \sharp (S\setminus S\cap W_s)=2$; we remark that $W_j\cap S=W_s\cap S $ for any $j,s $ thus we may call $E:=S\setminus S\cap W_s$. \\ 
    By Lemma \ref{lemmariformulato} we have that $ h^1(\mathcal{I}_E((1,\ldots ,1)-\varepsilon_s - \varepsilon_t))>0$, so $\sharp \pi_j(E)=1 $ for all $j\neq s,t$. 
    Since $E\subset H$ we have that $ \pi_i(E)=1$, moreover taking $s=1,2,3 $, if $t\neq j $, we get that  $\sharp E=1$, thus a contradiction. It remains to study what happens when $t=j $, i.e. if $\sharp (\pi_j(S\cap H))\geq 2 $. In such a case, when we let $s$ varies in  $\{ 1,\ldots , k \}\setminus \{i,j\}$, we get $\sharp \pi_s(S\cap H)=
    1$. Thus $\eta_j(S\cap H)=1 $, i.e. the three points of $S\cap H $ actually lies on a line, which is a contradiction with Remark \ref{retta}, because two of them are points of $A $ or $ B$.    
    
    \item\label{sharpleq2} Assume $\sharp (S\setminus S\cap H)\leq 2 $. Since $ h^1\big(\mathcal{I}_{S\setminus S \cap H}(\hat{\varepsilon}_i)\big)>0$, we get that $\sharp (S \setminus S\cap H)=2 $ and that $\sharp \eta_i(S\setminus S\cap H) =1$, which is a contradiction.
\end{enumerate}\vspace{-0,7cm}
\end{proof}

With these two lemmas we can conclude the case of two disjoint sets $A,B\in \mathcal{S}( 
Y,q)$ with $q$ of rank-3.

\begin{proposition}\label{6pP1P1P1} 
Let $q\in \sigma_3^0(\nu(Y))$  be a tensor of order $k\ge 5$ and let $\nu(Y)$ be its concise Segre. Then $\mathcal{S}( 
Y,q)$ does not contain  two disjoint sets.

\end{proposition}

\begin{proof}

By Lemma \ref{etainiettivi} there exists at least an index $i \in \{1,\dots, k \} $ such that  $\eta_{i\vert S}$ is injective, from which follows that the corresponding $\pi_{i\vert S} $ is also injective.
Now if $\eta_{j\vert S} $ is not injective for some $j\neq i $ then  $\pi_{i\vert S} $ is not injective, which is a contradiction with the assumption that $ \eta_{i\vert S}$ is injective.
Therefore thus  $\eta_{j|S}$ and $\pi_{j\vert S} $ have to be injective for all $j=1, \ldots , k$.

Write $A:=\{ a_1,a_2,a_3\} $ and $B:=\{ a_4,a_5,a_6 \} $. If the first factor is a $ \mathbb{P}^2$ take $L_1 \in \mathbb{P}^2 $ as a general line containing $ \pi_1(a_1)$ and define $H_1\in \vert \mathcal{I}_{a_1}(\varepsilon_1)\vert $ as $H_1:=\pi_1^{-1}L_1 $. For $i=2,\dots,5 $ take $ \{ H_i \}:=\vert \mathcal{I}_{a_i}(\varepsilon_i) \vert $ (this is possible since by hypothesis $k\geq 5$). Otherwise for all $ i=1,\dots,k$ take $ \{ H_i \}:=\vert \mathcal{I}_{a_i}(\varepsilon_i) \vert $. In both cases, since every  $\pi_{i\vert S}$ is injective we get that $H_1\cup \cdots \cup H_5 $ contains exactly $5$ points of $ S$. Thus from Lemma \ref{lemmariformulato} we get that $h^1\big(\mathcal{I}_{S\setminus (S\cap H_1\cup \cdots \cup H_5)}(0,0,0,0,0,1,\dots,1)\big)>0 $ which is a contradiction since $\sharp (S\setminus (S\cap H_1\cup \cdots \cup H_5))=1 $.

\end{proof}

\def\propfinale{

\begin{pp}
\begin{proposition}
Proposizione finale con tutto tutto.\\
usiamo $Y=\mathbb{P}^{n_1}\times \cdots \times \mathbb{P}^{n_k} $, with $k\geq 5 $,  $n_i\in \{1,2 \}$ a piacimento.
\begin{AB}
enunciato?
\end{AB}
\end{proposition}
\end{pp}

\begin{pp}
\begin{proof}
By Proposition \ref{6pP1P1P1} we exclude the case in which all factors have projective dimension $1 $, i.e. we exclude the case in which $q\in \sigma_3^0(\nu(Y))$, with $\nu(Y)$ the concise Segre of $q$ and $Y=(\mathbb{P}^1)^{\times k} $.\\
So we have that $n_1+\cdots +n_k >k $ because at least one between the  $n_i$'s must be $2$. By an argument similar to the one used in Lemma \ref{etainiettivi}, we have that at least one  $\eta_{i\vert S}$ is injective. Non so perchè ma ci riconduciamo al caso cinque fattori.\\
\begin{AB}
Questo è ricopiato identico da sopra... a che serve? dove si usa la dimensione 2 di uno dei fattori???
\end{AB} Thus by Lemma \ref{piiniettivo}, the corresponding $\pi_i$ is injective.
So we may say that $\pi_{i\vert S} $ is injective for all $ i$ such that $\eta_{i\vert S} $ is injective. But if $\eta_{i\vert S} $ is not injective for some $i\neq 5 $ then we get that $\pi_{5\vert S} $ is not injective, which is a contradiction with the injectivity of $ \eta_{5}$. Thus all $\pi_{i\vert S} $ are injective.\\ Write $A:=\{ a_1,a_2,a_3\} $ and $B:=\{ a_4,a_5,a_6 \} $ and take $ \{ H_i \}:=\vert \mathcal{I}_{a_i}(\varepsilon_i) \vert $ , for $i=1,\dots,5 $. Since every  $\pi_{i\vert S}$ is injective we get that $H_1\cup \cdots \cup H_5 $ contains exactly $5$ points of $ S$, contradicting the assumption $A,B\in \mathcal{S}(Y,q) $. questa ultima frase che ci porta a una contraddizione non l ho capita.

\end{proof}
\end{pp}
}

\section{Identifiability of rank-3 tensors}\label{sec.last}

The following theorem completely characterizes the identifiability of 
any rank-3 tensor and it is the main theorem of the present paper.

\begin{theorem}\label{main_theorem} Let $Y=\mathbb{P}^{n_1}\times \cdots \times \mathbb{P}^{n_k}$ be the multiprojective space of the concise Segre of a rank-3 tensor $q$. Denote with $\mathcal{S}(Y,q)$ the set of all subsets of $Y$ computing the rank of $q$. The rank-3 tensor $q$ is identifiable except in the following cases:
\begin{enumerate}
\item\label{1.} $q$ is a rank-3 matrix, in this case $\dim(\mathcal{S}(Y,q))=6$;
\item\label{2.} $q$ belongs to a tangent space of the Segre embedding of $Y=\mathbb{P}^1\times \mathbb{P}^1\times\mathbb{P}^1$, in this case $\dim(\mathcal{S}(Y,q))\geq 2 $;
\item\label{6.} $q$ is an order-4 tensor of $\sigma_3^0(Y)$ with $Y=\mathbb{P}^1\times\mathbb{P}^1\times\mathbb{P}^1\times\mathbb{P}^1$, in this case $\dim(\mathcal{S}(Y,q))\geq 1 $. 
\item\label{3.} $q$ is as in Example \ref{caso3} where $Y=\mathbb{P}^2\times\mathbb{P}^1\times\mathbb{P}^1$, in this case $\dim(\mathcal{S}(Y,q))=3 $;
\item\label{4.} $q$ is as in Example \ref{caso4} where $Y=\mathbb{P}^2\times\mathbb{P}^1\times\mathbb{P}^1$, in this case $\mathcal{S}(Y,q) $ contains two different $4 $-dimensional families;

\item\label{5.} $q $ is as in Proposition \ref{x1.1} where $ Y=\mathbb{P}^{n_1}\times \cdots \times \mathbb{P}^{n_k}$ is such that $k\ge 3$, $n_i\in \{1,2\}$ for $i=1,2$, $n_i=1$ for $i>2$ and $\sum _{i=1}^{k} n_i\ge 4$. In this case $\dim(\mathcal{S}(Y,q))\geq 2 $ and if $ n_1+n_2+k\geq 6$ then $\dim(\mathcal{S}(Y,q))=2 $. 

\end{enumerate}

\end{theorem}

\begin{proof}

In case \ref{1.}. the point $q$ is a rank-3 matrix therefore it is highly not-identifiable. See Remark \ref{caso1} for the computation of the dimension of $\mathcal{S}(Y,q)$.

Case \ref{2.}. is also well known: see \cite[Remark 3]{BalBer3}.

Case \ref{6.}. corresponds to the defective 3-th secant variety of the Segre embedding of $Y=(\mathbb{P}^1)^{\times 4}$ and the fact that all the elements of $\sigma_3^0(\nu(Y))$ are not-identifiable is shown in Remark \ref{allP14}. 
The fact that $\dim(\mathcal{S}(Y,q))=1 $ for the generic rank-3 tensor depends on the fact that the 3-th defect $\delta_3$ of  $\nu((\mathbb{P}^1)^{\times 4})$ is exactly 1 (cf. \cite{AOP}).  Moreover by \cite[Cap II, Ex 3.22, part (b)]{Hart} we get that for any rank $ 3$ tensor $ q$, the dimension $\dim(\mathcal{S}(Y,q))\geq 1 $.

Cases \ref{3.}., \ref{4.}. and \ref{5.}. are treated in Examples \ref{caso3} and \ref{caso4} and in Proposition \ref{x1.1}  
respectively.

\smallskip

All the above considerations prove that the list of cases enumerated in the statement corresponds to non indentifiable rank-3 tensors. We need to show that such a list is exhaustive. Since the matrix case is already fully covered by case \ref{1.} we only need to care about tensors of order at least 3. 

First of all recall that by Remark \ref{remarkconcision}, the concise Segre of a rank-3 tensor $q$  is $\nu(\mathbb{P}^{n_1} \times \cdots \times  \mathbb{P}^{n_k} )$, with $n_1, \dots ,n_k \in \{1,2 \}$. 
Then consider two distinct sets $A,B \in \mathcal{S}(Y,q)$. By Corollay \ref{solo5o6} it can only happen that $\sharp(A\cup B)=5,6$.

If $\sharp(A\cup B)=5$, the fact that our list of not-identifiable rank-3 tensors is exhaustive is proved in Propositions \ref{prop:P2P1i} and \ref{prop_P1}.

If $\sharp(A\cup B)=6$ we can firstly use Proposition \ref{6pP2P2Pn} to exclude the all the cases in which $Y$ has at least two factors of dimension 2. Then we start arguing by the number of factors of $Y$.
\\
If $Y$ has 3 factors and it is the product of $\mathbb{P}^1$'s only, then the unique tensors of rank-3 are those of the tangential variety to the Segre variety and this is case \ref{2.} of our theorem. The case of $Y=\mathbb{P}^{2}\times \mathbb{P}^{1}\times\mathbb{P}^1$ is completely covered by Proposition \ref{nonidesempi} together with Examples \ref{caso3} and \ref{caso4} (cf. Corollary \ref{P2P1P1}).
\\
If $Y$ has 4 factors and one of them is a $\mathbb{P}^2$, there is Proposition  \ref{6pP2P1P1P1P1} assuring that $\mathcal{S}(Y,q)$ does not contain two disjoint sets. If $Y$ is a product of four $\mathbb{P}^1$'s we are in case \ref{6.} of our theorem.
\\
The fact that if $Y$ has at least 5 factors then  $\mathcal{S}(Y,q)$ does not contain two disjoint sets is done in Proposition \ref{6pP1P1P1}.

\end{proof}

\section*{Acknowledgements}

The first two authors were partially supported by GNSAGA of INDAM.
We thanks the referee for the useful comments.

\end{document}